\documentclass[a4paper,12pt]{article}

\usepackage{hyperref}

\usepackage {amsfonts, amssymb, amscd,amsthm, amsmath, eucal}

\usepackage{amsbsy}
\usepackage[all]{xy}

\theoremstyle{plain}
{
  \newtheorem{thm}{Theorem}[section]
  \newtheorem{defn}[thm]{Definition}
  \newtheorem{cor}[thm]{Corollary}
  \newtheorem{lem}[thm]{Lemma}
  \newtheorem{prop}[thm]{Proposition}
  \newtheorem{rem}[thm]{Remark}
  \newtheorem{clm}[thm]{Claim}
  \newtheorem{notation}[thm]{Notation}
  
  \newtheorem{sssection}[thm]{}
}

{
}
\renewcommand{\subsubsection}{\sssection\rm}

\newcommand{\can}{\text{\rm can}}

\newcommand{\id}{\text{\rm id}}
\newcommand{\pr}{\text{\rm pr}}

\newcommand{\red}{\text{\rm red}}
\newcommand{\inc}{\text{\rm inc}}

\newcommand{\const}{\text{\rm const}}
\newcommand{\Spec}{\text{\rm Spec}}

\newcommand{\Ker}{\text{\rm Ker}}

\newcommand{\Aff}{\mathbf {A}}
\newcommand{\Pro}{\mathbf {P}}

\newcommand \xra {\xrightarrow }

\newcommand \hra {\hookrightarrow }

\newcommand{\ttf}{{\text{f}}}
\newcommand\mydim{\text{\rm dim}}

\renewcommand \id{\operatorname{id}}

\renewcommand \phi\varphi

\newcommand{\ad}{ad}
\newcommand{\et}{\text{\rm\'et}}
\newcommand{\R}{{\rm R}}

\newcommand{\ZZ}{\mathbb Z}

\begin{document}

\title
{On Grothendieck---Serre's conjecture concerning
principal $G$-bundles over reductive group schemes: I}

\author{I.~Panin\footnote{The author acknowledges support of the RFBR grant 13-01-00429-a}\ ,
A.~Stavrova\footnote{The author acknowledges support of
the J.E. Marsden postdoctoral fellowship of the Fields Institute for Research in Mathematical Sciences,
the RFBR grants 12-01-33057, 12-01-31100, 10-01-00551, 09-01-00878, and of the research program 6.38.74.2011 ``Structure theory and geometry
of algebraic groups and their applications in representation theory and algebraic K-theory'' at St.
Petersburg State University.}\ ,
N.~Vavilov\footnote{The author acknowledges support of the
RFFI projects 08-01-00756, 09-01-00762, 09-01-00784, 09-01-00878
and 09-01-91333} }

%\date{22.11.2012}

\maketitle

\begin{abstract}
%Given a simple simply connected group scheme $G$ over an essentially smooth local scheme $U$ over an infinite field
%and a principal $G$-bundle $\mathcal G$ trivial over the generic point of $U$
%we construct a principal $G$-bundle $\mathcal G_t$ over $\Aff^1_U$ which is trivial over
%$(\Aff^1_U)_f$ for a monic polinomial $f$ and such that the restriction of $\mathcal G_t$ restriction to $\{0\} \times U$ coincides with
%the original $G$-bundle $U$.

%If $G$ is additionally isotropic over $U$, we prove that every principal $G$-bundle $P_t$ over $\Aff^1_U$,
%which is trivial over
%$(\Aff^1_U)_f$ for a monic polinomial $f$
%is trivial itself.

%As a consequence the $G$-bundle $\mathcal G$ above is trivial over $U$.
%This confirms in positive a well-known conjecture of A.Grothendieck and J.-P.Serre.

%The recent solution of that conjecture for all reductive group schemes over semi-local regular domains
%containing an infinite field \cite{FP} is based significantly on the first mentioned result and on certain
%technical results proven in the second part of the present paper.

%The second result can not be improved since it false for anisotropic $G$.

Let $k$ be an infinite field. Let $R$ be the semi-local ring of a finite family of closed points on a $k$-smooth
affine irreducible variety and let $K$ be the fraction field of $R$ and let $G$ be a
reductive simple simply connected $R$-group scheme isotropic over $R$.
Our Theorem \ref{MainThmGeometric} states that for any Noetherian $k$-algebra $A$
the kernel of the map
$$ H^1_{\text{\'et}}(R \otimes_k A,G)\to H^1_{\text{\'et}}(K \otimes_k A,G) $$
\noindent
induced by the inclusion of $R$ into $K$ is trivial.
Theorem \ref{MainHomotopy} for $A=k$ and some other results of the present paper are used significantly in \cite{FP}
to prove the Grothendieck-Serre's conjecture for regular semi-local rings $R$ containing an infinite field.
%perfect subfield and let $K$ be its field of fractions. Let $G$ be a
%reductive semi-simple simply connected $R$-group scheme such that
%each of its $R$-indecomposable factors is isotropic.
%%%contains a split torus} $\Bbb G_{m,R}$.
%We prove that in this case the kernel of the map
%$$ H^1_{\text{\'et}}(R,G)\to H^1_{\text{\'et}}(K,G) $$
%\noindent
%induced by the inclusion of $R$ into $K$ is trivial. In other
%words, under the above assumptions every principal $G$-bundle $P$
%which has a $K$-rational point is itself trivial. This confirms a
%conjecture posed by Serre and Grothendieck. Our proof is based on
%a combination of methods of Raghunathan's paper \cite{R1},
%Ojanguren---Panin's paper \cite{OP1} and Panin's preprint
%\cite{Pa1}.
%\par
%If $R$ is the semi-local ring of several points on
%a $k$-smooth scheme, then it suffices to require that $k$ is infinite
%and keep the same assumption concerning $G$.

\end{abstract}

\section{Introduction}
\label{Introduction}

Recall that an $R$-group scheme $G$ is called reductive
(respectively, semi-simple or simple), if it is affine and smooth
as an $R$-scheme and if, moreover, for each ring homomorphism
$s:R\to\Omega(s)$ to an algebraically closed field $\Omega(s)$,
its scalar extension $G_{\Omega(s)}$
%% of $G$ via the homomorphism $s$
is a connected reductive (respectively, semi-simple or simple) algebraic
group over $\Omega(s)$. The class of reductive group schemes
contains the class of semi-simple group schemes which in turn
contains the class of simple group schemes. This notion of a
simple $R$-group scheme coincides with the notion of a {simple
semi-simple $R$-group scheme from Demazure---Grothendieck
\cite[Exp.~XIX, Defn.~2.7 and Exp.~XXIV, 5.3]{SGA3}.} {\it
Throughout the paper $R$ denotes an integral domain and $G$
denotes a semi-simple $R$-group scheme, unless explicitly stated
otherwise. All commutative rings that we consider are assumed to be Noetherian.}
\par
A semi-simple $R$-group scheme $G$ is called {\it simply connected\/}
(respectively, {\it adjoint\/}), provided that for an inclusion
$s:R\hra\Omega(s)$ of $R$ into an algebraically closed field
$\Omega(s)$ the scalar extension $G_{\Omega(s)}$ is a
simply connected (respectively, adjoint) $\Omega(s)$-group scheme.
This definition coincides with the one from \cite[Exp.~XXII.
Defn.~4.3.3]{SGA3}.
\par
%To each such $R$-group scheme $G$ and each ring homomorphism
%$s: R \to \Omega(s)$ to an algebraically closed field
%$\Omega(s)$ one can canonically assign the root system
%$\Phi(G_{\Omega(s)})$ of the group $G_{\Omega(s)}$.
%If $R$ is a domain, then the root system does not depend of
%a choice of that homomorphisms $s$. So, for a domain $R$ we call
\par
A well-known conjecture due to J.-P.~Serre and A.~Grothendieck
\cite[Remarque, p.31]{Se},
\cite[Remarque 3, p.26-27]{Gr1},
and
\cite[Remarque 1.11.a]{Gr2}
asserts that given a regular local ring $R$ and its field of
fractions $K$ and given a reductive group scheme $G$ over $R$ the
map
$$ H^1_{\text{\'et}}(R,G)\to H^1_{\text{\'et}}(K,G), $$
\noindent
induced by the inclusion of $R$ into $K$, has trivial kernel. The
following theorem, which is one of the main result of the present paper,
asserts that for simple and simply connected isotropic group schemes over certain rings $R$
this is indeed the case
({\it recall that a simple
$R$-group scheme is called isotropic if it contains a split torus}
$\Bbb G_{m,R}$). Actually, we prove something significantly stronger. Namely,

\begin{thm}
\label{MainThmGeometric} Let $k$ be an infinite field. Let
$\mathcal O$ be the semi-local ring of finitely many closed points on a
$k$-smooth irreducible affine $k$-variety $X$ and let $K$ be its field of
fractions.
Let $G$ be an isotropic simple simply connected group
scheme over $\mathcal O$.
Then for any Noetherian $k$-algebra $A$
% containing a split rank $1$ torus $\Bbb G_m$.
the map
$$ H^1_{\text{\rm\'et}}(\mathcal O \otimes_k A,G)\to H^1_{\text{\rm\'et}}(K \otimes_k A,G), $$
\noindent
induced by the inclusion $\mathcal O$ into $K$, has trivial kernel.
\end{thm}
In other words, under the above assumptions on $\mathcal O$ and $G$ each
principal $G$-bundle $P$ over $\mathcal O \otimes_k A$ which is trivial over $K \otimes_k A$
is itself trivial.
In the case $A=k$ the main result of \cite{FP} is much stronger,
since there are no anisotropy assumptions on $G$ there. However, it seems that there is no reason
to expect that Theorem \ref{MainThmGeometric} holds for anisotropic $G$.

Theorem~\ref{MainThm} extends easily to the case of simply connected semi-simple
group schemes, using the Faddeev---Shapiro lemma. However, in this generality its statement
is a bit more technical and we postpone it till
Section~\ref{SectSemi-Simple} (see Theorem~\ref{MainThmGeometricSemi-Simple}).
All other results stated below extend to semi-simple simply connected group schemes as well.

Theorem \ref{MainThmGeometric} is an easy consequence of the following two theorems.
\begin{thm}
\label{MainHomotopy}
Let $k$, $\mathcal O$, $K$, $A$ be the same as in Theorem \ref{MainThmGeometric}.
Let $G$ be a not necessarily isotropic simple simply connected group
scheme over $\mathcal O$.
%Let $k$ be an infinite field. Let
%$\mathcal O$ be the semi-local ring of finitely many {\bf closed} points on a
%smooth irreducible $k$-variety $X$ and let $K$ be its field of
%fractions.
%Let $G$ be an isotropic simple simply connected group
%scheme over $\mathcal O$.
%Then for any Noetherian $k$-algebra $A$
% containing a split rank $1$ torus $\Bbb G_m$.
Let $\mathcal G$ be a principal $G$-bundle over $\mathcal O \otimes_k A$ which is trivial
over $K \otimes_k A$.
Then there exists a principal $G$-bundle $\mathcal G_t$ over
$\mathcal O[t] \otimes_k A$
and a monic polynomial
$f(t) \in \mathcal O[t]$
such that
\par
(i) the $G$-bundle $\mathcal G_t$ is trivial over $(\mathcal O[t]_f) \otimes_k A$,
\par
(ii) the evaluation of $\mathcal G_t$ at $t=0$ coincides
%$\mathcal G$-bundle $(\mathcal G_t)|_{U \times \{0\} \times_{Spec k} Spec A}$
with the original $G$-bundle $\mathcal G$,
\par
(iii) $f(1) \in \mathcal O$ is invertible in $\mathcal O$.

%the map
%$$ H^1_{\text{\rm\'et}}(\mathcal O \otimes_k A,G)\to H^1_{\text{\rm\'et}}(K \otimes_k A,G), $$
%\noindent
%induced by the inclusion $\mathcal O$ into $K$, has trivial kernel.
\end{thm}

\begin{thm}\label{MainThmArithmetic}
Let $k$ be a not necessarily infinite field.
Let $B$ be a Noetherian $k$-algebra. Let $G$ be an isotropic simple simply connected group
scheme over $B$, that is $G$ contains a torus $\mathbb G_{m,B}$.
%$k$, $\mathcal O$, $K$, $G$, $A$ be the same as in Theorem \ref{MainThmGeometric}.
Let $P_t$ be a principal $G$-bundle over
$B[t]$
and let
$h(t) \in B[t]$
a monic polynomial
such that
\par
(i) the $G$-bundle $P_t$ is trivial over $B[t]_h$,
\par
(ii) $h(1) \in B$ is invertible in $B$.\\
Then the principal $G$-bundle $P_t$ is trivial.
%\par
%(ii) the $\mathcal G$-bundle $(\mathcal G_t)|_{U \times \{0\} \times_{Spec k} Spec A}$
%coincides with the original $G$-bundle $\mathcal G$.
\end{thm}
\begin{rem}
To prove Theorem \ref{MainThmGeometric} one needs to substitute in Theorem \ref{MainThmArithmetic}
$B=\mathcal O \otimes_k A$, $P_t:=\mathcal G_t$, $h(t)=f(t) \otimes 1$
from Theorem \ref{MainHomotopy}. By Theorem
\ref{MainThmArithmetic}
the $G$-bundle $\mathcal G_t$ is trivial.
Now by the item (ii) of Theorem \ref{MainHomotopy} the original $G$-bundle $\mathcal G$ is trivial.
\end{rem}
\begin{rem}
Theorem \ref{MainThmArithmetic} does not hold in the case of anisotropic $G$ even for $B=\mathcal O$
with $\mathcal O$ as in Theorem \ref{MainHomotopy}.
There are various conterexamples. Only a weaker form of Theorem \ref{MainThmArithmetic} holds in the case of anisotropic $G$ and $B=\mathcal O$
as it is proved in~\cite[Thm. 2]{FP}. This is why we are skeptical that Theorem \ref{MainThmGeometric} holds in the anisotropic case.
\end{rem}
\begin{thm}
\label{MainThm}
Let $R$ be a regular semi-local domain containing an infinite field and let $K$ be
the fraction field of $R$.
Let $G$ be an isotropic simple simply connected group
scheme over $R$,
containing a split rank $1$ torus $\Bbb G_{m,R}$.
Then for any Noetherian commutative ring $A$
the map
$$ H^1_{\text{\rm\'et}}(R \otimes_{\ZZ} A,G)\to H^1_{\text{\rm\'et}}(K \otimes_{\ZZ} A,G), $$
\noindent
induced by the inclusion $R$ into $K$, has trivial kernel.
\end{thm}

%\begin{thm}
%\label{MainThm} Let $R$ be regular semi-local domain containing an
%infinite perfect field and let $K$ be its field of fractions. Let $G$ be
%an isotropic simple simply connected group scheme over $R$.
%% containing a split rank $1$ torus $\Bbb G_m$.
%Then the map
%$$ H^1_{\text{\rm\'et}}(R[t_1, \dots, t_n],G)\to H^1_{\text{\rm\'et}}(K[t_1, \dots, t_n],G), $$
%\noindent
%induced by the inclusion $R$ into $K$, has trivial kernel.
%\end{thm}

%More generally, it is natural to ask, whether
%two elements $\xi,\zeta\in H^1_{\text{\rm\'et}}(\mathcal O \otimes_k A,G)$ which are
%equal over $K \otimes_k A$ are already equal over $\mathcal O \otimes_k A$.
%In general, this does
%not follow from Theorem \ref{MainThm}. However, this is indeed the
%case provided at least one of the group schemes $G(\xi)$ or
%$G(\zeta)$ is isotropic.
%???Query: Does it imply that two $G$-bundles
%isomorphic over $K$ are indeed isomorphic. N.V.???

Combining Theorem~\ref{MainThm} with the well-known result of M.S. Raghunathan and A. Ramanathan~\cite{RR}, we obtain
the following

\begin{cor}\label{HominvarCor}
Let $R$ be a regular domain containing $\mathbb{Q}$.
Let $G$ be an isotropic simple simply connected group
scheme over $R$, containing a split rank $1$ torus $\Bbb G_{m,R}$.
Then the map
$$
H^1_{\et}(R[t],G)\to H^1_{\et}(R,G),
$$
induced by evaluation at $t=0$, has trivial kernel.
\end{cor}

To put these statements into context, let us recall other known results on the Serre---Grothen\-dieck conjecture.
\par\smallskip
$\bullet$ The case where the group scheme $G$ comes from the
ground field $k$ is completely solved by
J.-L.~Colliot-Th\'el\`ene, M.~Ojanguren, M.~S.~Raghunatan and
O.~Gabber: in \cite{C-TO} and \cite{R1}, \cite{R2} when
$k$ is infinite;
%in \cite{C-TO} when $k$ is infinite and PERFECT; in \cite{R1}, \cite{R2} when
%$k$ is INFINITE;
O.~Gabber \cite{Ga} announced a proof for an
arbitrary ground field $k$.
\par\smallskip
$\bullet$ The case of an arbitrary reductive group scheme over a
discrete valuation ring is completely solved by Y.~Nisnevich in
\cite{Ni}.
\par\smallskip
$\bullet$ The case where $G$ is an arbitrary torus over a regular
local ring was settled by J.-L.~Colliot-Th\'{e}l\`{e}ne and
J.-J.~Sansuc in \cite{C-T-S}.
\par\smallskip
$\bullet$
For most simple group schemes of classical series the Serre---Grothendieck conjecture was solved in
works of the first author, A. Suslin, M. Ojanguren and K. Zainoulline
\cite{PS}, \cite{OP1}, \cite{Z}, \cite{OPZ}.
%\cite{Pa1}.
In fact, unlike our Theorem \ref{MainThm}, {\it no isotropy
hypotheses} was imposed there.
%However, for the {\it exceptional\/}
%group schemes {\it our theorem is new\/}.
%Our proof is based on different ideas and treats classical and exceptional types in a unified way.
%\par\smallskip
%$\bullet$ There exists a folklore result, concerning a simple $R$-group scheme of type
%$G_2$, where $R$ is a regular semi-local ring containing an infinite perfect field or
%the semilocal ring of finitely many points on a $k$-smooth scheme with an infinite field $k$.
%That result gives affirmative answer in this case and also independent
%of isotropy hypotheses, see %% the very end of
%the paper by V.~Chernousov and the first author \cite{ChP}.
\par\smallskip
$\bullet$
%In \cite{Pa2} the first author used the Serre---Grothendieck conjecture part of Theorem
%\ref{MainThm}, i.e. the case $A=k$, together with the main result of
%\cite{C-T-S}, to establish the Serre---Grothendieck conjecture for an
% {\it any reductive} group scheme satisfying a similar isotropy
%condition. Later
The first author, the second author, and V. Petrov
%used this theorem to establish
proved the Serre---Grothendieck
conjecture for strongly inner adjoint groups of type $E_6$ or $E_7$~\cite{PPS}, and for groups of type
$F_4$ with trivial $f_3$ invariant~\cite{PS-f4}, under the same assumptions on $R$.
\par\smallskip
$\bullet$
V. Chernousov~\cite{Ch} established the Serre---Grothendieck conjecture for groups of type $F_4$ with
trivial $g_3$ invariant, under the assumption that $R$ is a regular local ring containing a field of charactersitic $0$.
\par\smallskip
$\bullet$
In \cite{Pa2} the first author reduced the Serre---Grothendieck conjecture to the case of semi-simple simply connected
group schemes (assuming that $R$ is the semilocal ring of finitely many closed points on a $k$-smooth affine variety with
an infinite field $k$).
\par\smallskip
$\bullet$
R. Fedorov and the first author in~\cite{FP} prove the Serre---Grothendieck conjecture
for an arbitrary reductive group scheme over a semilocal regular ring containing an infinite field.
Their work relies heavily on the results of the present paper and of~\cite{Pa2}.

\medskip
The authors would like to thank Vladimir Chernousov, 
Philippe Gille, Victor Petrov, and Konstantin Pimenov for useful discussions on the 
subject of the present paper.

\section{Almost elementary fibrations}
\label{ElementaryFibrations}

In this Section we modify a result of M. Artin from~\cite{LNM305} concerning existence of nice neighborhoods.
The following notion is a modification of the one introduced by Artin in~\cite[Exp. XI, D\'ef. 3.1]{LNM305}.
\begin{defn}
\label{DefnElemFib} An \emph{almost elementary fibration} over a scheme $S$ is a morphism of
schemes $p:X\to S$ which can be included in a commutative diagram
\begin{equation}
\label{SquareDiagram}
    \xymatrix{
     X\ar[drr]_{p}\ar[rr]^{j}&&
\overline X\ar[d]_{\overline p}&&Y\ar[ll]_{i}\ar[lld]_{q} &\\
     && S  &\\    }
\end{equation}
of morphisms satisfying the following conditions:
\begin{itemize}
\item[{\rm(i)}] $j$ is an open immersion dense at each fibre of
$\overline p$, and $X=\overline X-Y$;
\item[{\rm(ii)}]
$\overline p$ is smooth projective all of whose fibres are geometrically
irreducible of dimension one;
\item[{\rm(iii)}] $q$ is a finite flat morphism all of whose fibres are non-empty;
\item[\rm(iv)] the morphism $i$ is a closed embedding and
the ideal sheaf
$I_Y \subset \mathcal O_{\overline X}$ defining the closed subscheme $Y$ in $\overline X$
is locally principal.
\end{itemize}
\end{defn}

\begin{rem}
This definition is motivated by the following example. Take a field $k$ and $S=\Spec(k)$, take
$\overline X= \Pro^1_k$, take a closed point $y \in \Pro^1_k$ and set $X=\Pro^1_k - \{y\}$, $Y=y$.
Then the structure morphism $X \to S$ is an almost elementary fibration. If the field extension
$k(y)/k$ is purely inseparable, then $X \to S$ is not an elementary fibration in the sense of
Artin~\cite[Exp. XI, D\'ef. 3.1]{LNM305}.
\end{rem}

%The following Bertini type theorem is an extension of Artin's
%result~\cite[Exp. XI, Thm. 2.1]{LNM305}.

%\begin{thm}
%\label{GeneralSection} Let $k$ be an infinite field, and let
%$V\subset\Pro^n_k$ be a locally closed subscheme of pure
%dimension $r$. Further, let $V^{\prime}\subset V$ be an open
%subscheme consisting of all points $x\in V$ such that $V$ is
%$k$-smooth at $x$. Finally, let $p_1,p_2,\dots,p_m\in\Pro^n_k$ be
%a family of pairwise distinct closed points.
%%% with $p_i \neq p_j$ for $i \neq j$.
%For a family $H_1(d),H_2(d),\dots,H_s(d)$, with $s\le r$,
%, where $s\leq r$, of
%of hyperplanes of degree $d$ containing all points $p_i$, $1\le
%i\le m$, set
%$$ Y=H_1(d)\cap H_2(d)\cap\dots\cap H_s(d). $$

%Then there exists an integer $d$ depending on the family
%$p_1,p_2,\dots,p_m$ such that if the family
%$H_1(d),H_2(d),\dots,H_s(d)$ with $s\leq r$ is sufficiently
%general, then $Y$ crosses $V$ transversally at each point of $Y
%\cap V^{\prime}$.
%\par
%If, moreover, $V$ is irreducible {\rm(}respectively, geometrically
%irreducible{\rm)} and $s<r$, then for the same integer $d$ and for
%a sufficiently general family $H_1(d),H_2(d),\dots,H_s(d)$ the
%intersection $Y\cap V$ is irreducible {\rm(}respectively,
%geometrically irreducible{\rm)}.
%\end{thm}
%\begin{proof}
%See Appendix,~\ref{ArtinsNeighb}.
%\end{proof}

%Using this theorem,
We prove the following result, which is a
slight modification of Artin's result~\cite[Exp. XI, Prop. 3.3]{LNM305}.

\begin{prop}
\label{ArtinsNeighbor} Let $k$ be an infinite field, $X$ be a
smooth geometrically irreducible variety over $k$, $x_1,x_2,\dots,x_n\in X$
be closed points. Then there exists a Zariski open neighborhood
$X^0$ of the family $\{x_1,x_2,\dots,x_n\}$ and an almost elementary
fibration $p:X^0\to S$, where $S$ is an open subscheme of the
projective space $\Pro^{\mydim X-1}$.
\par
If, moreover, $Z$ is a closed co-dimension one subvariety in $X$,
then one can choose $X^0$ and $p$ in such a way that $p|_{Z\bigcap
X^0}:Z\bigcap X^0\to S$ is finite surjective.
\end{prop}

The proofs of the above Proposition and of the following one are provided in Appendix,~\ref{ArtinsNeighb}.

\begin{prop}
\label{CartesianDiagram} Let $p: X \to S$ be an almost elementary
fibration. If $S$ is a regular semi-local irreducible scheme, then there
exists a commutative diagram of $S$-schemes
\begin{equation}
\label{RactangelDiagram}
    \xymatrix{
X\ar[rr]^{j}\ar[d]_{\pi}&&\overline X\ar[d]^{\overline \pi}&&
Y\ar[ll]_{i}\ar[d]_{}&\\
\Aff^1\times S\ar[rr]^{\text{\rm in}}&&\Pro^1\times S&&
\ar[ll]_{i}\{\infty\}\times S &\\  }
\end{equation}
\noindent
such that the left hand side square is Cartesian. Here $j$ and
$i$ are the same as in Definition \ref{DefnElemFib}, while $\pr_S
\circ\pi=p$, where $\pr_S$ is the projection $\Aff^1\times S\to
S$.

In particular, $\pi:X\to\Aff^1\times S$ is a finite surjective
morphism of $S$-schemes, where $X$ and $\Aff^1\times S$ are
regarded as $S$-schemes via the morphism $p$ and the projection
$\pr_S$, respectively.
\end{prop}

\section{Nice triples}
\label{NiceTriples}

In the present section we introduce and study certain collections
of geometric data and their morphisms. The concept of a {\it nice
triple\/} is very similar to that of a {\it standard triple\/}
introduced by Voevodsky \cite[Defn.~4.1]{Vo}, and was in fact
inspired by the latter notion. Let $k$ be an infinite field, let $X/k$
be a smooth geometrically irreducible variety, and let
$x_1,x_2,\dots,x_n\in X$ be its closed points. Further, let
$\mathcal O=\mathcal O_{X,\{x_1,x_2,\dots,x_n\}}$ be the
corresponding geometric semi-local ring.
%$U:= \text{Spec}(\mathcal O_{X, \{x_1,x_2,\dots,x_n\}})$.
\begin{defn}
\label{DefnNiceTriple} Let $U:=\text{\Spec}(\mathcal
O_{X,\{x_1,x_2,\dots,x_n\}})$.
%$J= \cap^n_{i=1} \mathfrak m_i, l=\mathcal O_{X, \{x_1,x_2,\dots,x_n\}}/J$.
A \emph{nice triple} over $U$ consists of the following data:
\begin{itemize}
\item[\rm{(i)}] a smooth morphism $q_U:\mathcal X\to U$, where $\mathcal X$ is an irreducible scheme,
\item[\rm{(ii)}] an element $f\in\Gamma(\mathcal X,\mathcal
O_{\mathcal X})$,
%a closed codimension one subscheme
%$\mathcal Z$,
\item[\rm{(iii)}] a section $\Delta$ of the morphism $q_U$,
\end{itemize}
subject to the following conditions:
\begin{itemize}
\item[\rm{(a)}]
%{\bf each component of each fibre of the morphism $q_U$
%is geometrically irreducible of dimension one,}
each irreducible component of each fibre of the morphism $q_U$
has dimension one,
%and can be decomposed as
%$pr_U \circ \Pi$,
%where
%$\Pi: \mathcal X \to \Aff^1 \times U$
%is a finite surjective morphism;
\item[\rm{(b)}]
the module
$\Gamma(\mathcal X,\mathcal O_{\mathcal X})/f\cdot\Gamma(\mathcal X,\mathcal O_{\mathcal X})$
is finite as
a $\Gamma(U,\mathcal O_{U})=\mathcal O$-module,
%$\mathcal Z$ is finite and surjective over $U$;
\item[\rm{(c)}] there exists a finite surjective $U$-morphism
$\Pi:\mathcal X\to\Aff^1\times U$, \item[\rm{(d)}]
$\Delta^*(f)\neq 0\in\Gamma(U,\mathcal O_{U})$.
\end{itemize}
\end{defn}
\begin{defn}
\label{DefnMorphismNiceTriple} A \emph{morphism} of two nice
triples
$(q^{\prime}: \mathcal X^{\prime} \to U,f^{\prime},\Delta^{\prime})\to(q: \mathcal X \to U,f,\Delta)$
is
an \'{e}tale morphism of $U$-schemes $\theta:\mathcal
X^{\prime}\to\mathcal X$ such that
\begin{itemize}
\item[\rm{(1)}] $q^{\prime}_U=q_U\circ\theta$, \item[\rm{(2)}]
$f^{\prime}=\theta^{*}(f)\cdot h^{\prime}$ for an element
$h^{\prime}\in\Gamma(\mathcal X^{\prime},\mathcal O_{\mathcal
X^{\prime}})$,
%$\theta^{-1} (\mathcal Z) \subset \mathcal Z^{\prime}$
%$\theta^{-1} (\mathcal Z)$ is finite surjective over $U$);
\item[\rm{(3)}] $\Delta=\theta\circ\Delta^{\prime}$.
\end{itemize}
\end{defn}
Two observations are in order here.
\par\smallskip
$\bullet$ Item (2) implies in particular that $\Gamma(\mathcal
X^{\prime},\mathcal O_{\mathcal
X^{\prime}})/\theta^*(f)\cdot\Gamma(\mathcal X^{\prime},\mathcal
O_{\mathcal X^{\prime}})$ is a finite
$\mathcal O$-module.
\par\smallskip
$\bullet$ It should be emphasized that no conditions are imposed
on the interrelation of $\Pi^{\prime}$ and $\Pi$.
\par\smallskip

%The proof of this theorem is given in Section
%\ref{SecEquatingGroups}. \\\\

Let $U$ be as in Definition
\ref{DefnNiceTriple}. Let $(\mathcal X,f,\Delta)$ be a nice triple
over $U$. Then for each finite surjective $U$-morphism
$\sigma:\mathcal X\to\Aff^1\times U$ and the corresponding
$\mathcal O$-algebra inclusion
$\mathcal O[t]\hra\Gamma(\mathcal X,\mathcal O_{\mathcal X})$
%the element $t$ is a non-zero divisor in
%$\Gamma(\mathcal X, \mathcal O_{\mathcal X})$,
the algebra $\Gamma(\mathcal X,\mathcal O_{\mathcal X})$ is
finitely generated as an $\mathcal O[t]$-module. Since both rings
$\mathcal O[t]$ and $\Gamma(\mathcal X,\mathcal O_{\mathcal X})$
are regular, the algebra $\Gamma(\mathcal X,\mathcal O_{\mathcal
X})$ is finitely generated and projective as an $\mathcal
O[t]$-module by theorem \cite[Cor.~18.17]{E}.
Let
$T^r - a_{n-1}T^{r-1} + \dots \pm N(f)$
be the characteristic polynomial
of the
$\mathcal O[t]$-module endomorphism
$\Gamma(\mathcal X,\mathcal O_{\mathcal X}) \xra{f} \Gamma(\mathcal X,\mathcal O_{\mathcal X})$,
and set
\begin{equation}
g_{f,\sigma}:=f^{r-1}-a_{n-1}f^{r-2}+\dots\pm a_1\in
\Gamma(\mathcal X,\mathcal O_{\mathcal X}).
\end{equation}
\noindent
\begin{lem}
\label{FandG} $f\cdot g_{f,\sigma}=\pm N(f)\in\Gamma(\mathcal X,\mathcal O_{\mathcal X})$.
\end{lem}
\begin{proof}
Indeed, the characteristic polynomial of the operator
$\Gamma(\mathcal X,\mathcal O_{\mathcal X}) \xra{f}
\Gamma(\mathcal X,\mathcal O_{\mathcal X})$ vanishes on $f$.
\end{proof}
Let us state two crucial results which will be used in our main
construction. Their proofs are given in Sections
\ref{SectElemNisnevichSquare}
and
\ref{SecEquatingGroups}
respectively.

\begin{thm}
\label{ElementaryNisSquare} Let $U$ be as in Definition
\ref{DefnNiceTriple}. Let $(\mathcal X,f,\Delta)$ be a nice triple
over $U$, such that $f$ vanishes at every closed point of $\Delta(U)$. There exists a distinguished finite surjective morphism
$$ \sigma:\mathcal X\to\Aff^1\times U $$
\noindent
of $U$-schemes which enjoys the following properties.
\begin{itemize}
\item[\rm{(1)}] $\sigma$ is \'{e}tale along the closed subset
$\{f=0\}\cup\Delta(U)$. %%
\item[\rm{(2)}] For $g_{f,\sigma}$ and $N(f)$ defined by the
distinguished $\sigma$, one has
$$ \sigma^{-1}\Big(\sigma\big(\{f=0\}\big)\Big)=
\{N(f)=0\}=\{f=0\}\sqcup\{g_{f,\sigma}=0\}. $$
%%{\rm(the disjoint union )
%follows:
%let $N(f)$ is the determinant of
%take the characteristic polynomial
%$t^r - a_{n-1}t^{r-1} + \dots \pm N(f)=0$
%of
%the $\mathcal O[t]$-module endomorphism
%$\Gamma(\mathcal X, \mathcal O_{\mathcal X}) \xra{f} \Gamma(\mathcal X, \mathcal O_{\mathcal X})$
%and set
%$g:=f^{r-1}-a_{n-1}f^{r-2} + \dots \pm a_1 \in \Gamma(\mathcal X, \mathcal O_{\mathcal X})$.
%take the characteristic polynomial of that endomorphism vanishes on $f$, which means that
%$f^r -a_{n-1}f^{r-1} + \dots \pm N(f)=0$
%and
%$N(f)= f \cdot g$
%for an element
%$g \in \mathcal O[t]$.
\item[\rm{(3)}] Denote by $\mathcal X^0\hra\mathcal X$ the largest
open subscheme where the morphism $\sigma$ is \'{e}tale. Write
$g$ for $g_{f,\sigma}$ in this item. Then the square
\begin{equation}
\label{SquareDiagram2}
    \xymatrix{
     \mathcal X^{0}_{N(f)}=\mathcal X^{0}_{fg}  \ar[rr]^{\inc} \ar[d]_{\sigma^0_{fg}} &&
     \mathcal X^{0}_g \ar[d]^{\sigma^0_g}  &\\
     (\Aff^1 \times U)_{N(f)} \ar[rr]^{\inc} && \Aff^1 \times U &\\
    }
\end{equation}
is an elementary Nisnevich square. More precisely, this square is
Cartesian and the morphism of the reduced closed subschemes
$$ \sigma^0_g|_{\{f=0\}_{\red}}:\{f=0\}_{\red}\to\{N(f)=0\}_{\red} $$
\noindent
of the schemes $\mathcal X^{0}_g$ and $\Aff^1\times U$ is an
isomorphism.
\item[\rm{(4)}] One has $\Delta(U)\subset\mathcal X^{0}_g$.
\end{itemize}
\end{thm}
%The proof of this Theorem is given in Section
%\ref{SectElemNisnevichSquare}.
%\par
%Using Theorems \ref{ThEquatingGroups} and \ref{ElementaryNisSquare}
%in the next Section we construct
%% \ref{MainConstruction}
%data (a) to (d) from the Introduction subject to Conditions (1*)
%to (6*).
%%%%%%%%%%%%%%%%%%%%%%%%%%%%%%%%%%%%%%%%%%%%%%%%%%%%%%%%%%%%%%%%%%
%\label{SectElemNisnevichSquare}
\begin{rem}
One readily sees that if in Theorem~\ref{ElementaryNisSquare} we let $\mathcal X^0$ be any open subscheme of $\mathcal X$
such that $\sigma$ is \'etale on $\mathcal X^0$ and $\mathcal X^0$ contains the closed subset $\{f=0\}\cup\Delta(U)$,
then all the claims of this theorem are still valid. In particular, if needed,
one can assume that $\mathcal X^0$ is an affine scheme.
\end{rem}

\begin{thm}
\label{ThEquatingGroups} Let $U$ be as in Definition
\ref{DefnNiceTriple}. Let $(\mathcal X,f,\Delta)$ be a nice triple
over $U$. Let $G_{\mathcal X}$ be a simple simply connected
$\mathcal X$-group scheme, and let $G_U:=\Delta^*(G_{\mathcal
X})$. Finally, let $G_{\const}$ be the pull-back of $G_U$ to
$\mathcal X$. Then there exists a morphism $\theta:(\mathcal
X^{\prime},f^{\prime},\Delta^{\prime})\to(\mathcal X,f,\Delta)$ of
nice triples and an isomorphism
$$ \Phi: \theta^*(G_{\const}) \to \theta^*(G_{\mathcal X}) $$
\noindent
of $\mathcal X^{\prime}$-group schemes such that
$(\Delta^{\prime})^*(\Phi)=\id_{G_U}$.
\end{thm}

\section{Proof of Theorem~\ref{ElementaryNisSquare}}
\label{SectElemNisnevichSquare}
The nearest aim is to
prove Theorem~\ref{ElementaryNisSquare}. We will use
analogues of three lemmas from~\cite{Pa1} making them
characteristic free. Lemma \ref{Lemma2} is a refinement of
\cite[Lemma 5.2]{OP2}.
%and Lemma
%\ref{Lemma1}
%is a refinement of Quillen's trick.

\begin{lem}
\label{Lemma3} Let $k$ be an infinite field
and let $S$ be an
$k$-smooth equidimensional $k$-algebra of dimension one.
%and let $S$ be an
%$k$-smooth domain of dimension $1$ such that $k$ is algebraically closed in $S$.
Let $f\in
S$ be a non-zero divisor.
\par
Let $\mathfrak m_0$ be a maximal ideal with $S/\mathfrak m_0=k$.
Let $\mathfrak m_1,\mathfrak m_2,\dots,\mathfrak m_n$ be pairwise
distinct maximal ideals of $S$ {\rm(}possibly $\mathfrak
m_0=\mathfrak m_i$ for some $i${\rm)}. Then there exists a
non-zero divisor $\bar s\in S$ such that $S$ is finite over
$k[\bar s]$ and
\begin{itemize}
\item[{\rm(1)}] the ideals $\mathfrak n_i:=\mathfrak m_i\cap
k[\bar s]$, $1\le i\le n$, are pairwise distinct. If $\mathfrak
m_0$ is distinct from all $\mathfrak m_i$'s, then $\mathfrak n_0:=\mathfrak m_0\cap k[\bar
s]$ is distinct from all $\mathfrak n_i$'s; %%
\item[{\rm(2)}] the extension $S/k[\bar s]$ is \'etale at each
$\mathfrak m_i$, $i=1,2,\ldots,n$, and at $\mathfrak m_0$; %%
\item[{\rm(3)}] $k[\bar s]/\mathfrak n_i=S/\mathfrak m_i$ for each
$i=1,2,\ldots,n$; %%
\item[{\rm(4)}] $\mathfrak n_0=\bar s k[\bar s]$.
\end{itemize}
\end{lem}
\begin{proof}
Let $x_i$, $0\le i\le n$, be the points on $\Spec(S)$ corresponding
to the ideals $\mathfrak m_i$. Consider a closed embedding
$\Spec(S)\hra\Aff^n_k$ and find a generic linear projection
$p:\Aff^n_k\to\Aff^1_k$, defined over $k$ and such that  the following holds:
\begin{itemize}
\item[{\rm(1)}] for all $i,j\geq 0$ one has $p(x_i)\neq p(x_j)$,
provided that $x_i\neq x_j$; %%
\item[{\rm(2)}] for each index $i\geq 0$ the map $p|_{\Spec(S)}:
\Spec(S)\to\Aff^1_k$ is \'{e}tale at the point $x_i$; %%
\item[{\rm(3)}] for each $i$, the separable degree of the extension
$k(x_i)/k(p(x_i))$ is one. %%
\end{itemize}
These items imply equalities $k(p(x_i))=k(x_i)$, for all $i$.
Indeed, the extension $k(x_i)/k(p(x_i))$ is separable by (2). By
(3) we conclude that $k(p(x_i))=k(x_i)$. Lemma follows.

\end{proof}

\begin{lem}
\label{Lemma4} Under the hypotheses of Lemma \ref{Lemma3} let
$f\in S$ be a non-zero divisor which does not belong to a maximal
ideal distinct from $\mathfrak m_0,\mathfrak m_1,\dots,\mathfrak
m_n$. Let $\bar s\in S$ be an element satisfying {\rm(1)} to {\rm(4)} of Lemma~\ref{Lemma3}.
Let $N(f)=N_{S/k[\bar s]}(f)$ be the norm of $f$. Then one has
\begin{itemize}
\item[{\rm(a)}] $N(f)=fg$ for an element $g\in S$; %%
\item[{\rm(b)}] $fS+gS=S$; %%
\item[{\rm(c)}] the map $k[\bar s]/(N(f))\to S/(f)$ is an
isomorphism.
\end{itemize}
\end{lem}
\begin{proof}
Straightforward.

\end{proof}

\begin{lem}
\label{Lemma2}
Let $k$ be an infinite field, and let $R$ be a domain which is a semi-local essentially smooth
$k$-algebra with maximal ideals $\mathfrak p_i$, $1\le i\le m$.
Let $A \supseteq R[t]$ be another domain, smooth as an $R$-algebra and
finite over $R[t]$.
Assume that for each $i$ the $R/\mathfrak
p_i$-algebra $A_i=A/\mathfrak p_i A$ is equidimensional of
dimension one.
%Assume that for each $i$ the
%$R/\mathfrak p_i$-algebra
%$A_i=A/\mathfrak p_i A$ is a domain and $R/\mathfrak p_i$ is algebraically closed in $A_i$.
%equidimensional of
%dimension one.
%$R/\mathfrak m_i$
%is a domain and assume there exists an element
%$h \in R \setminus \mathfrak m R$
%such that $R_h$ is $A$-smooth.
Let $\epsilon:A\to R$ be an $R$-augmentation and
$I=\Ker(\epsilon)$. Given an $f\in A$ with
$$ 0\neq\epsilon(f)\in\bigcap\limits^m_{i=1}\mathfrak p_i \subset R$$
\noindent
and such that the $R$-module $A/fA$ is finite, one can find an
element $u\in A$ satisfying the following conditions:
\begin{itemize}
\item[{\rm(1)}] $A$ is a finite projective module over $R[u]$; %%
\item[{\rm(2)}] $A/uA=A/I\times A/J$ for some ideal $J$; %%
\item[{\rm(3)}] $J+fA=A$; %%
\item[{\rm(4)}] $(u-1)A+fA=A$; %%
\item[{\rm(5)}] set $N(f)=N_{A/R[u]}(f)$, then $N(f)=fg\in A$ for
some $g\in A$; %%
%(by \cite[Cor.18.17]{E} the $R[u]$-module $A$ is a projective module
%since $A$ and $R[u]$ are regular and $A$ is finite over $R[u]$);
\item[{\rm(6)}] $fA+gA=A$; %%
\item[{\rm(7)}] the composition map $\varphi:R[u]/(N_{A/R[u]}(f))\to A/(N_{A/R[u]}(f))\to A/(f)$ is an isomorphism.
\end{itemize}
\end{lem}

\begin{proof}
Replacing $t$ by $t-\epsilon(t)$ we may assume that
$\epsilon(t)=0$. Since $A$ is finite over $R[t]$, it follows from a theorem of
Grothendieck~\cite[Cor. 17.18]{E} that it is a finite
projective $R[t]$-module.

Since $A$ is finite over $R[t]$ and $A/fA$ is finite over $R$ we
conclude that $R[t]/(N_{A/R[t]}(f))$ is finite over $R$, and hence
$R/(tN_{A/R[t]}(f))$ is finite over $R$.

Setting $v=tN_{A/R[t]}(f)$, we get an
integral extension $R[t]$ over $R[v]$.
Thus $A$ is finite over $R[v]$.
By the theorem of
Grothendieck~\cite[Cor. 17.18]{E} $A$ is a finite
projective $R[v]$-module.
% is an integral extension.

Applying Lemma \ref{FandG} to $A$ over $R[t]$ (not over $R[v]$) one gets an equality
$N_{A/R[t]}(f)=f\cdot g_{f,t} \in A$
for an element $g_{f,t} \in A$.
Thus
$$ v=t\cdot N_{A/R[t]}(f)=t\cdot f\cdot g_{f,t} \in fA, \ \ \text{and} \ \ \epsilon (v)=\epsilon(t)\cdot \epsilon(N_{A/R[t]}(f))=0$$
%Now replacing
%$t$ by $u$ we may assume
%we have $\epsilon (v)=\epsilon(t)\epsilon(fg)=0$ and  $v\in fA$.
\par
Below, we use the bar $\bar{\ }$ to denote reduction modulo an ideal, and the
subscript $i$ to indicate that reduction is modulo $\mathfrak p_iA$, $1\le i\le m$. Let
$l_i=\bar R_i=R/\mathfrak p_i$. By the assumption of the lemma, the
$l_i$-algebra $\bar A_i$ is $l_i$-smooth
%domain
equidimensional
of dimension $1$.
The element $\bar f_i \in \bar A_i$ is a non-zero divisor since
$\bar A_i/\bar f_i \bar A_i=\overline {(A/fA)}_i$ is a finite $l_i$-module.
%such that
%$l_i$ is algebraically closed in $\bar A_i$.
Let
${\mathfrak m}^{(i)}_1,{\mathfrak m}^{(i)}_2,\dots,{\mathfrak m}^{(i)}_{n_i}$
be distinct maximal ideals
of $\bar A_i$ dividing $\bar f_i$ and let $\mathfrak
m^{(i)}_0=\Ker(\bar{\epsilon}_i)$. Let $\bar s_i\in\bar A_i$ be
such that the extension $\bar{A}_i/l_i[\bar s_i]$ satisfies
conditions (1) to (4) of Lemma~\ref{Lemma3}.
\par
Let $s\in A$ be a common lifting of $\bar s_i$'s, in other words,
$\overline s=\bar s_i$ in $\bar A_i$ for all $i=1,\ldots,m$. Replacing $s$
by $s-\epsilon(s)$ we may assume that $\epsilon(s)=0$ and, as
above, $\overline s=\bar s_i$ for all $i=1,\ldots,m$.
\par
Let $s^n+p_1(v)s^{n-1}+\dots+p_n(v)=0$ be an integral dependence relation
for $s$. Let $N$ be an integer larger than
$\max\{2,\deg(p_j(t))\}$, where $j=1,2,\dots,n$. Then for any
$r\in k^{\times}$ the element $u=s-rv^N$ has the following property:
$v$ is integral over $R[u]$. Thus, for any $r\in k^{\times}$ the ring
$A$ is integral over $R[u]$.
\par
On the other hand, one has
$\bar v_i\in\mathfrak m^{(i)}_j$
for all
$1\le i\le m$ and all $0\le j\le n_i$,
since $v\in fA$ and $\epsilon(v)=0$.
It implies that each element
$\bar u_i=\bar s_i-r\bar{v_i}^N$ still satisfies conditions (1) to (4)
of Lemma~\ref{Lemma3}.
\par
We claim that the element $u\in R$  has all the properties listed in the statement of the present lemma,
for almost all $r\in k^{\times}$.
\par
Indeed, for almost all $r\in k^{\times}$ the element $u$ satisfies
Conditions (1) to (4) of Lemma \ref{Lemma2}.
%\cite[Lemma 5.2]{OP2}.
It remains to
show that Conditions (5) to (7) hold for all $r\in k^{\times}$.
\par
Since $A$ is finite over $R[u]$, the same theorem of Grothendieck~\cite[Cor.~17.18]{E} implies that it is a finite projective
$R[u]$-module. To prove (5), consider the characteristic
polynomial of the operator $A\xra{f} A$ as an $R[u]$-module
operator. This polynomial vanishes on $f$ and its free term equals
$\pm N_{A/R[u]}(f)$, the norm of $f$. Thus,
$f^n-a_1f^{n-1}+\dots\pm N_{A/R[u]}(f)=0$ and $N_{A/R[u]}(f)=f\cdot g_{f,u}$
for some $g_{f,u}\in A$.
\par
To prove (6), one has to verify that the above $g$ is a unit
modulo the ideal $fA$. It suffices to check that for each index
$i$ the element $\bar g_i\in\bar A_i$ is a unit modulo the ideal
$\bar f_i\bar A_i$. To that end observe that the field
$l_i=R/\mathfrak p_i$, the $l_i$-algebra $S_i=\bar A_i$, its
maximal ideals $\mathfrak m^{(i)}_0,\mathfrak
m^{(i)}_1,\dots,\mathfrak m^{(i)}_{n_i}$ and the element $\bar u_i$
satisfy the hypotheses of Lemma \ref{Lemma4}, with $u$ replaced by
$\bar u_i$. Now, by Item (b) of Lemma \ref{Lemma4} the reduction
$\bar g_i$ is a unit modulo the ideal $\bar f_i\bar R_i$.
\par
To prove (7), observe that $R[u]/(N_{A/R[u]}(f))$ and $A/fA$ are
finite $R$-modules. Thus, it remains to check that the map
$\varphi:R[u]/(N_{A/R[u]}(f))\to A/fA$ is an isomorphism modulo
each maximal ideal $\mathfrak p_i$. To that end it
suffices to verify that the map $\bar\varphi_i:l_i[\bar
u_i]/(N(\bar f_i))\to\bar A_i/\bar f_i\bar A_i$ is an isomorphism
for each index $i$, where $N(\bar f_i):=N_{\bar A_i/l_i[\bar
u]}(\bar f_i)$. Now, by Item (c) of Lemma \ref{Lemma4} the map
$\bar\varphi_i$ is an isomorphism. This finishes the proof.%Lemma follows.

\end{proof}

%%%%%%%%%%%%%%%%%%%%%%%%%%%%%%%%%%%%%%%%%%%%%%%%%%%%%%%%%%%%%%%%%%%

%\begin{cor}
%\label{NormAtZeroAndOne} Under the hypotheses of Lemma
%$\ref{Lemma2}$ let $K$ be the quotient field of $R$,
%$A_K=A\otimes_R K$ and $\epsilon_K=\epsilon\otimes_R\id:A_K\to K$.
%Consider the inclusion $K[s]\subset A_K$. Then the norm $N(f)\in
%K[s]$ does not vanish at the points $1$ and $0$ of the affine line
%$\Aff^1_K$.
%\end{cor}

%\begin{proof}
%Condition (4) of \ref{Lemma2} implies that $N(f)$ does not vanish
%at the point $1$. Since $\epsilon_K (f)\neq 0\in K$, Conditions
%(2) and (3) imply that $N(f)$ does not vanish at $0$ either.
%\end{proof}

%Now we are all set to finish the proof of Theorem
%\ref{ElementaryNisSquare}.

\begin{proof}[Proof of Theorem \ref{ElementaryNisSquare}]
Let $U=\text{Spec}(\mathcal O_{X,\{x_1,x_2,\dots,x_r\}})$ be as in
Definition \ref{DefnNiceTriple}. Write $R$ for $\mathcal
O_{X,\{x_1,x_2,\dots,x_r\}}$.
It is a domain which is a semi-local
essentially smooth $k$-algebra with maximal ideals $\mathfrak
p_i$, $1\le i\le r$. Let $(\mathcal X,f,\Delta)$ be a nice triple
over $U$. We show that it gives rise to certain data subject to
the hypotheses of Lemma \ref{Lemma2}.
\par
Let $A=\Gamma(\mathcal X,\mathcal O_{\mathcal X})$. It is a domain, since $\mathcal X$ is irreducible.
It is an
$R$-algebra via the ring homomorphism $q^*_U:R\to\Gamma(\mathcal
X,\mathcal O_{\mathcal X})$. Furthermore, it is smooth as an
$R$-algebra. The triple $(\mathcal X,f,\Delta)$ is a nice triple.
Thus, there exists a finite surjective $U$-morphism $\Pi:\mathcal
X\to\Aff^1_U$. It induces an $R$-algebra inclusion
$R[t]\hra\Gamma(\mathcal X,\mathcal O_{\mathcal X})=A$ such that
$A$ is finitely generated as an $R[t]$-module. Also, for all $i=1,\ldots,r$, the $R/\mathfrak p_i$-algebra
$A/\mathfrak p_i A$ is equidimensional of dimension one. Let
$$
\epsilon=\Delta^*:A\to R
$$ be an $R$-algebra homomorphism
induced by the section $\Delta$ of the morphism $q_U$. Clearly,
this $\epsilon$ is an augmentation; set $I=\Ker(\epsilon)$.
Further, since $(\mathcal X,f,\Delta)$ is a nice triple,
$\epsilon(f)\neq 0\in R$ and $A/fA$ is finite as an $R$-module.
Finally, $f$ vanishes at every closed point of $\Delta(U)$ by the assumption of the Theorem.
%% since $(\mathcal X,f,\Delta)$.
Summarising the above, we conclude that we are in
the setting of Lemma \ref{Lemma2}, and may use the conclusion of
that Lemma.
\par
Thus, there exists an element $u\in A$ subject to Conditions (1)
through (7) of Lemma \ref{Lemma2}. This $u$ induces an $R$-algebra
inclusion $R[u]\hra A$ such that $A$ is finite as an
$R[u]$-module. Let
$$ \sigma:\mathcal X\to\Aff^1\times U $$
\noindent
be the $U$-scheme morphism induced by the above inclusion $R[u]\hra
A$. Clearly, $\sigma$ is finite and surjective.
In the rest of the proof we write $t$ instead of $u$, and consider $A$ as an $R[t]$-module via $\sigma$.
Let
$N(f):=N_{A/R[t]}(f) \in R[t]\subseteq A$ and $g_{f,\sigma}\in A$ be the elements defined
just above Lemma~\ref{FandG}.
%in Theorem
%\ref{ElementaryNisSquare}.
\par
We claim that this morphism $\sigma$ and the chosen
elements $N(f)$ and $g_{f,\sigma}$ satisfy conclusions (1) to (4) of Theorem
\ref{ElementaryNisSquare}. Let us verify this claim. Since $A$ is
finite as an $R[t]$-module and both rings $R[t]$ and $A$ are
regular, the $R[t]$-module $A$ is finitely generated and
projective, see \cite[Corollary 18.17]{E}. Thus, $\sigma$ is
\'{e}tale at a point $x\in \mathcal X$ if and only if the
$k(\sigma(x))$-algebra $k(\sigma(x))\otimes_{R[t]} A$ is
\'{e}tale. If the point $x$ belongs to the closed subscheme
$\text{Spec}(A/\mathfrak p_i A)$ for some maximal ideal $\mathfrak p_i$ of $R$, then
$$ k(\sigma(x))\otimes_{R[t]}A=k(\sigma(x))\otimes_{(R/\mathfrak
p_i)[t]}A/\mathfrak p_iA. $$
\noindent
We can conclude that $\sigma$ is \'{e}tale at a specific point $x$
if and only if the $(R/\mathfrak p_i)[t]$-algebra $A/\mathfrak
p_iA$ is \'etale at the point $x$. It follows from the proof of
Lemma \ref{Lemma2} that the morphism $\sigma$ induces a morphism
$\text{Spec}(A/\mathfrak p_i A)\xra{\sigma_i}\Aff^1_{l_i}$ on the
closed fibre $\text{Spec}(A/\mathfrak p_i A)$ for each $i$. This
induced morphism is \'{e}tale along the vanishing locus of the
function $\bar f_i$
and along each point $\overline
\Delta_i(\Spec\,l_i)$.
Indeed, for the vanishing locus of the function
$\bar f_i$
this follows from Items (6) and (7) of Lemma~\ref{Lemma2}. It follows from the hypotheses of
Lemma~\ref{Lemma2}
that the function $f$ vanishes at each maximal ideal containing $I$.
Thus $\sigma$ is \'{e}tale along the closed subscheme
$\mathcal X$ defined by the ideal $I$, that is along $\Delta(U)$.
This settles Item (1) of Theorem~\ref{ElementaryNisSquare}.
\par
Consider Item (2). Write $g$ for $g_{f,\sigma}$. The first of the following equalities
$$ \sigma^{-1}(\sigma(\{f=0\}))=\{N(f)=0\}=\{f=0\}\sqcup\{g=0\} $$
\noindent
is a commonplace. %% general well-known fact.
The second one follows from the equality $N(f)=\pm f \cdot g$, proved in Lemma~\ref{FandG}
%just above Theorem
%\ref{ElementaryNisSquare}
and Item (6) of Lemma~\ref{Lemma2}.
\par
Clearly, the square
(\ref{SquareDiagram2}) is Cartesian and the morphism $\sigma^0_g$
is \'{e}tale. The scheme $\mathcal X^{0}_g$ contains a closed
subscheme $\Delta(U)$, and hence is non-empty. Item (7) of Lemma
\ref{Lemma2} shows that the morphism of the reduced closed
subschemes
$$ \sigma^0_g|_{\{f=0\}_{\red}}: \{f=0\}_{\red} \to \{N(f)=0\}_{\red} $$
\noindent
is an isomorphism. Thus, we have checked Item (3) of Theorem
\ref{ElementaryNisSquare}.
\par
It remains only to check Item (4). We already know that
$\{f=0\}\subset\mathcal X^{0}_g$. Both schemes $\Delta(U)$ and $\{f=0\}$
are semi-local and the set of closed points of $\Delta(U)$ is
contained in the set of closed points of the closed set $\{f=0\}$ by the assuptions of the theorem.
Thus, $\Delta(U)\subset\mathcal X^{0}_g$. This concludes the proof
of Item (4) of Theorem \ref{ElementaryNisSquare} and thus of the
theorem itself.

\end{proof}

\section{Proof of Theorem \ref{ThEquatingGroups}}
\label{SecEquatingGroups}
The aim of this Section is to
proof of Theorem \ref{ThEquatingGroups}.
We begin with the
following Proposition which
is a straightforward analogue of \cite[Prop. 7.1]{OP1}
\def\tilde{\widetilde}
\begin{prop}
\label{PropEquatingGroups} Let $S$ be a regular semi-local
irreducible scheme and let $G_1,G_2$ be two semi-simple
simply-connected $S$-group schemes which are twisted forms of each other. Further, let $T\subset S$ be a
closed sub-scheme of $S$ and $\varphi:G_1|_T\to G_2|_T$ be an
$S$-group scheme isomorphism. Then there exists a finite \'{e}tale
morphism $\tilde S\xra{\pi}S$ together with its section
$\delta:T\to\tilde S$ over $T$ and an $\tilde S$-group scheme
isomorphism $\Phi:\pi^*{G_1}\to\pi^*{G_2}$ such that
$\delta^*(\Phi)=\varphi$.
\end{prop}

Since the proof of the Proposition~\ref{PropEquatingGroups} is rather long we first give an outline.
Clearly, $G_1$ and $G_2$ are of the same type.
By~\cite[Exp. XXIV, Cor. 1.8]{SGA3} there exists an $S$-scheme $Isom_S(G_1,G_2)$ representing
the functor that sends an $S$-scheme $W$ to the set of
all $W$-group scheme isomorphisms from $W \times_S G_1$ to $W \times_S G_2$.
The isomorphism $\varphi$
from the hypothesis of Proposition
\ref{PropEquatingGroups}
determines a section
$\delta: T \to Isom_S(G_1,G_2)$
of the structure map
$Isom_S(G_1,G_2) \to S$.
By Lemmas \ref{tildeS}
and \ref{IsomG_1G_2} below
there exists a closed subscheme $\tilde S$ of
$Isom_S(G_1,G_2)$
which is finite \'{e}tale over $S$ and contains $\delta(T)$.
So, we have a commutative diagram of $S$-schemes
\begin{equation}
\label{D}
    \xymatrix{
    T \ar[rrd]_-{i} \ar[rr]^-{\delta}&& \tilde S \ar[rr]^-{} \ar[d]^-{\pi}  && Isom_S(G_1,G_2)  \ar[lld]^-{} \\
    && S  &&.  \\
    }
\end{equation}
such that the horizontal arrows are closed embeddings.
Thus we get an isomorphism
$\Phi: \pi^*(G_1) \to \pi^*(G_2)$
such that
$\delta^{*}(\Phi)=\varphi$.

The precise proof of the Proposition requires some auxiliary results and
will be given right below Lemma \ref{tildeS}.
%postponed till the very end of the Section.
Clearly, $G_1$ and $G_2$ are of the same type. Let $G_{0}$ be a split
semi-simple
simply connected algebraic group over the ground field $k$ such that
$G_1$ and $G_2$ are twisted forms of the $S$-group scheme
$S \times_{Spec(k)} G_0$.
Let
$Aut_k(G_0)$
be the automorphism scheme of the algebraic $k$-group $G_0$.
It is known that
$Aut_k(G_0)$
is a semi-direct product of the algebraic $k$-group
$G^{\ad}_0$ and a finite group,
where
$G^{\ad}_0$ is a group adjoint to $G_0$.
Also, $Aut_k(G_0)$ is a smooth affine algebraic $k$-group (for example, by~\cite[Exp. XXIV, Cor. 1.8]{SGA3}).
Set for short
$Aut:=Aut_k(G_0)$
and
$Aut_S$ for the $S$-group scheme
$S \times_{Spec(k)} Aut$.

Consider an $S$-scheme $Isom_S(G_{0,S},G_2)$ constructed in
\cite[Exp.  XXIV, Cor.1.8]{SGA3}
and representing a functor that sends an $S$-scheme $W$ to the set of
all $W$-group scheme isomorphisms
$\varphi_2:W \times_S G_{0,S}\to W \times_S G_2$.
Similarly, consider
an $S$-scheme $Aut_S(G_2)$ constructed in
\cite[Exp. XXIV, Cor. 1.8]{SGA3}
and representing a functor that sends an $S$-scheme $W$ to the set of
all $W$-group scheme automorphisms
$\alpha: W \times_S G_2 \to W \times_S G_2$.

The functor transformation
$(\varphi_2, \alpha_2) \mapsto \varphi_2 \circ \alpha^{-1}_2$
defines an $S$-scheme morphism
$$Isom_S(G_{0,S},G_2) \times_S Aut_S \to Isom_S(G_{0,S},G_2)$$
which makes the $S$-scheme
$Isom_S(G_{0,S},G_2)$
a principal right $Aut_S$-bundle.
The functor transformation
$(\beta_2, \varphi_2) \mapsto \beta_2 \circ \varphi_2$
defines an $S$-scheme morphism
$$Aut_S(G_2) \times_S Isom_S(G_{0,S},G_2) \to Isom_S(G_{0,S},G_2)$$
which makes the $S$-scheme
$Isom_S(G_{0,S},G_2)$
a principal left $Aut_S(G_2)$-bundle.

Analogously,
the functor transformation
$(\alpha_1, \varphi_1) \mapsto \alpha_1 \circ \varphi_1$
makes the $S$-scheme
$Isom_S(G_1,G_{0,S})$
a principal left $Aut_S$-bundle and
the functor transformation
$( \varphi_1, \beta_1) \mapsto \varphi_1 \circ \beta_1$
makes the $S$-scheme
$Isom_S(G_1,G_{0,S})$
a principal right $Aut_S(G_1)$-bundle.

Let $_{2}P_r$ be a left principal $Aut_S(G_2)$-bundle and at the same time a right principal $Aut_S$-bundle
such that the two actions commute.
Let $_{l}P_1$ be a left principal $Aut_S$-bundle and at the same time a right principal $Aut_S(G_1)$-bundle
such that the two actions commute.
Let $Y$ be a $k$-variety equipped with a left and a right $Aut_k$-actions which commute. Then the $k$-scheme
$$({_{2}P_r}) \times_S (Y_S) \times_S ({_{l}P_1 })$$
is equipped with a left $Aut_k \times Aut_k$-action given by
$$
(\alpha_2, \alpha_1)(p_2,y,p_1)=(p_2 \alpha^{-1}_2, \alpha_2 y \alpha^{-1}_1, \alpha_1 p_1).
$$
The orbit space does exist (it can be constructed by descent). Denote it by $_{2}Y_1$.
We now show that it is an $S$-scheme. Indeed, the structure morphism
$Y \to Spec(k)$
defines a morphism
$$({_{2}P_r}) \times_S (Y_S) \times_S ({_{l}P_1 }) \to ({_{2}P_r}) \times_S ({_{l}P_1 })$$
respecting the $Aut \times Aut$-actions on both sides. Thus it defines a morphism of the orbit spaces
$$_{2}Y_1 \to (_{2}Spec(k)_1)=S.$$
The latter equality holds since
$({_{2}P_r}) \times_S ({_{l}P_1 })$
is a principal left $Aut \times Aut$-bundle with respect the left action given by
$(\alpha_2, \alpha_1)(p_2,p_1)=(p_2 \alpha^{-1}_2, \alpha_1 p_1)$.

The construction
$Y \mapsto \ _{2}Y_1$
has several nice properties. Namely,
\begin{itemize}
\item[(i)]
it is natural with respect to $k$-morphisms of $k$-varieties
$Y \to Y^{\prime}$
commuting with the given two-sided $Aut \times Aut $-actions on $Y$ and $Y^{\prime}$,
\item[(ii)]
it takes closed embeddings to closed embeddings,
\item[(iii)]
it takes open embeddings to open embeddings,
\item[(iv)]
it takes $k$-products to $S$-products,
\item[(v)]
locally in the \'{e}tale topology on $S$, the $S$-schemes
$Y_S$ and $_{2}Y_1$ are isomorphic.
\end{itemize}
Set
${_{2}P_r}= Isom_S(G_{0,S},G_2)$
and
${_{l}P_1}= Isom_S(G_1,G_{0,S})$.
The functor transformation
$(\varphi_2, \alpha, \varphi_1) \mapsto \varphi_2 \circ \alpha \circ \varphi_1$
gives a morphism of representable $S$-functors
$$Isom_S(G_{0,S},G_2) \times_S (Aut_S) \times_S Isom_S(G_1,G_{0,S}) \xra{\displaystyle \Phi} Isom_S(G_1,G_2).$$
The equality
$$\varphi_2 \circ \alpha \circ \varphi_1=
(\varphi_2 \circ \alpha^{-1}_2)\circ (\alpha_2 \circ \alpha \circ \alpha^{-1}_1) \circ (\alpha_1 \circ \varphi_1)$$
shows that the morphism $\Phi$ induces a morphism
$\bar \Phi: \ _{2}(Aut)_1 \to Isom_S(G_1,G_2)$.

\begin{lem}
\label{IsomG_1G_2}
The $S$-morphism
$$\bar \Phi: \  _{2}(Aut)_1 \to Isom_S(G_1,G_2)$$
is an isomorphism.
\end{lem}

\begin{proof}
It suffices to prove that $\bar \Phi$ is an isomorphism locally in the \'{e}tale topology on $S$.
The latter follows from the property (v).
\end{proof}

Now let $G_0$ and $Aut$ be as above. There is a closed embedding of algebraic groups
$\rho: Aut \hra GL_{V,k}$
for an $n$-dimensional $k$-vector space $V$.
Replacing
$\rho$ with $\rho \oplus det^{-1} \circ \rho$ we get a closed embedding of algebraic $k$-groups
$\rho_1: Aut \hra SL_{W,k}$,
where $W=V \oplus k$.
Let $End:=End_k(W)$.
Clearly, the composition
$in: Aut \xra{\rho_1} SL_{W,k} \hra End$
is a closed embedding.
We will identify
$Aut$ with its image in $End$.
Let $\overline {Aut}$ be the closure of $Aut$ in the projective space $\Pro(k \oplus End )$.
Set
$Aut_{\infty}:=\overline {Aut}- Aut$
regarded as a reduced scheme. So, we get a commutative diagram of $k$-varieties
\begin{equation}
\label{RactangelDiagram}
    \xymatrix{
Aut \ar[rr]^{j}\ar[d]^{in}&& \overline {Aut} \ar[d]^{\overline {in}}&& Aut_{\infty} \ar[ll]_{i}\ar[d]^{in_{\infty}}&\\
End \ar[rr]^{\text{\rm J}}&& \Pro(k \oplus End ) && \ar[ll]_{I} \Pro(End)  &\\  }
\end{equation}
where the left square is Cartesian. All varieties are equipped with the left
$Aut \times Aut$-action induced by
$Aut \times Aut$-action on the affine space $k \oplus End$
given by
$(g_1,g_2)(\alpha,c)= (c, g_1 \alpha g^{-1}_2)$.
All the arrows in this
diagram respect this action. Applying to this diagram the above construction
$Y \mapsto \ _{2}Y_1$,
we obtain a commutative diagram of $S$-schemes
\begin{equation}
\label{RactangelDiagram-S}
    \xymatrix{
_{2}Aut_1 \ar[rr]^{j}\ar[d]^{in}&& _{2}(\overline {Aut})_1 \ar[d]^{\overline {in}}&& _{2}(Aut_{\infty})_1 \ar[ll]_{i}\ar[d]^{in_{\infty}}&\\
_{2}End_1  \ar[rr]^{\text{\rm J}}&& \Pro(\mathcal O_S \oplus \ _{2}End_1 ) && \ar[ll]_{I} \Pro(_{2}End_1)  &\\  }
\end{equation}
where the square on the left is Cartesian.

From now on we assume that $S$ is a semi-local irreducible scheme. Then the vector bundle
$_{2}End_1$ is trivial.
Since it is trivial, we may choose homogeneous coordinates $Y_i$'s on
$\Pro(\mathcal O_S \oplus \ _{2}End_1)$
such that the closed subschemes
$\{Y_0=0 \}$
and
$\Pro(_{2}End_1)$
of the scheme
$\Pro(\mathcal O_S \oplus \ _{2}End_1)$
coincide and the $S$-scheme
$\Pro(\mathcal O_S \oplus \ _{2}End_1)$
itself is isomorphic to the projective space
$\Pro^{n^2}_S$. Thus the diagram
(\ref{RactangelDiagram-S})
of $S$-schemes and of $S$-scheme morphisms can be rewritten as follows
\begin{equation}
\label{Compactification}
    \xymatrix{
_{2}Aut_1 \ar[rr]^{j}\ar[d]^{in}&& _{2}(\overline {Aut})_1 \ar[d]^{\overline {in}}&& _{2}(Aut_{\infty})_1 \ar[ll]_{i}\ar[d]^{in_{\infty}}&\\
\{Y_0 \neq 0 \}   \ar[rr]^{\text{\rm J}}&& \Pro^{n^2}_S && \ar[ll]_{I} \{Y_0 = 0 \}  &\\  }
\end{equation}
where the square on the left is Cartesian.
Since
$_{2}(Aut_{\infty})_1= \ _{2}(\overline {Aut})_1 - \ _{2}Aut_1$,
the set-theoretic intersection
$_{2}(\overline {Aut})_1 \cap \{Y_0 = 0 \}$
in $\Pro^{n^2}_S$ coincides with
$_{2}(Aut_{\infty})_1$.

The following Lemma is the lemma~\cite[Lemma 7.2]{OP1}.
\begin{lem}
\label{Lemma7_2}
Let $S=Spec(R)$ be a regular semi-local scheme and $T$ a closed
subscheme of $S$. Let $\bar X$ be a closed subscheme of
$\Pro^{N}_S= Proj(S[Y_0,\dots,Y_{N}])$ and
$X=\bar X\cap\Aff^{N}_S$,
where
$\Aff^{N}_S$
is the affine space defined by
$Y_0\neq0$. Let
$X_{\infty}=\bar X\setminus X$ be the intersection of $\bar X$ with the
hyperplane at infinity
$Y_0=0$. Assume further that
\begin{itemize}
\item[(1)]$X$ is smooth and equidimensional over $S$,
of relative dimension $r$.
\item[(2)]
For every closed
point $s\in S$ the closed fibres of $X_\infty$ and $X$
satisfy
$$\dim (X_\infty(s))< \dim (X(s))=r\;.$$
\item[(3)]
Over $T$ there exists a section $\delta:T\to X$
of the canonical projection $X\to S$.
\end{itemize}
Then there exists a closed subscheme $\tilde S$ of $X$ which is finite
\'etale
  over
$S$ and contains $\delta(T)$.
\end{lem}
The diagram
(\ref{Compactification})
shows that the $S$-schemes
$X= \ _{2}Aut_1$,
$\bar X = \ _{2}(\overline {Aut})_1$
and
$X_{\infty}= \ _{2}(Aut_{\infty})_1$
satisfy all the hypotheses of Lemma
\ref{Lemma7_2}
except possibly the conditions (2) and (3). To check (2), observe that the diagram of $S$-schemes
\begin{equation}
\label{RactangelDiagram}
    \xymatrix{
_{2}Aut_1  \ar[rr]^{\text{\rm j}}&& _{2}(\overline {Aut})_1  && \ar[ll]_{i} \ _{2}(Aut_{\infty})_1  &\\  }
\end{equation}
locally in the \'{e}tale topology on $S$ is isomorphic to the diagram of $S$-schemes
\begin{equation}
\label{RactangelDiagram}
    \xymatrix{
Aut \times S \ar[rr]^{\text{\rm j}}&& (\overline {Aut}) \times S  && \ar[ll]_{i} (Aut_{\infty}) \times S.  &\\  }
\end{equation}
This follows from the property (v) of the construction $Z \to {}_{2}Z_1$.
Since $Aut$ is equidimensional and $\overline {Aut}$ is the closure of $Aut$ in $\Pro(End \oplus k)$, one has
$$\dim (Aut_{\infty}) < \dim (\overline {Aut})= \dim Aut.$$
Thus the assumption (2) of Lemma~\ref{Lemma7_2}
is fulfilled. Whence we have proved the following

\begin{lem}
\label{tildeS}
Assume $S$ is a regular semi-local irreducible scheme and assume we are given with
a closed subscheme $T \subset S$ equipped with a section
$\delta: T \to \ _{2}Aut_1 $
of the structure map
$_{2}Aut_1 \to S$. Then there exists a closed subscheme
$\tilde S$ of $_{2}Aut_1$
which is finite and \'{e}tale over $S$ and contains $\delta(T)$.
\end{lem}

\begin{proof}[Proof of Proposition \ref{PropEquatingGroups}]
By Lemma
\ref{IsomG_1G_2}
the $S$-schemes
$Isom_S(G_1,G_2)$
and
$_{2}Aut_1$
are naturally isomorphic as $S$-schemes.
The isomorphism
$\varphi$
from the hypotheses of the Proposition
\ref{PropEquatingGroups}
determines a section
$\delta: T \to Isom_S(G_1,G_2)= \ _{2}Aut_1$
of the structure map
$Isom_S(G_1,G_2)= \ _{2}Aut_1 \to S$.
By Lemma
\ref{tildeS}
there exists a closed subscheme $\tilde S$ of
$_{2}Aut_1= Isom_S(G_1,G_2)$
which is finite \'{e}tale over $S$ and contains $\delta(T)$.
So, we have morphisms (even closed inclusions) of $S$-schemes
\begin{equation}
\label{D}
    \xymatrix{
    T \ar[rrd]_-{i} \ar[rr]^-{\delta}&& \tilde S \ar[rr]^-{} \ar[d]^-{\pi}  && Isom_S(G_1,G_2)  \ar[lld]^-{} \\
    && S  &&.  \\
    }
\end{equation}
Thus we get an isomorphism
$\Phi: \pi^*(G_1) \to \pi^*(G_2)$
such that
$\delta^{*}(\Phi)=\varphi$.

\end{proof}

{\bf Proof of Theorem \ref{ThEquatingGroups}.} We can start by
almost literally repeating arguments from the proof of \cite[Lemma
8.1]{OP1}, which involve the following purely geometric lemma
\cite[Lemma 8.2]{OP1}.
\par
For reader's convenience below we state that Lemma adapting
notation to the ones of Section \ref{NiceTriples}.
%and
%\cite[Lemma 8.3]{OP}.
%One should replace notation as follows:
%the finite surjective map
%$q: \mathcal X \to \Aff^1 \times U$
%from
%\cite[Lemma 8.2]{OP}
%replace with our finite surjective map
%$\Pi: \mathcal X \to \Aff^1 \times U$
%from the condition of Theorem
%\ref{EquatingGroups}.
%The following lemma is proved in
%\cite[Lemma 8.2]{OP}
Namely, let $U$ be as in Definition \ref{DefnNiceTriple} and let
$(\mathcal X,f,\Delta)$ be a nice triple over $U$. Further, let
$G_{\mathcal X}$ be a simple simply-connected $\mathcal X$-group
scheme, $G_U:=\Delta^*(G_{\mathcal X})$, and let $G_{\const}$ be
the pull-back of $G_U$ to $\mathcal X$. Finally, by the definition
of a nice triple there exists a finite surjective morphism
$\Pi:\mathcal X\to\Aff^1\times U$ of $U$-schemes.
\begin{lem}
\label{Lemma_8_2} Let $\mathcal Y$ be a closed nonempty sub-scheme
of $\mathcal X$, finite over $U$. Let $\mathcal V$ be an open
subset of $\mathcal X$ containing $\Pi^{-1}(\Pi(\mathcal Y))$. There
exists an open set $\mathcal W \subseteq \mathcal V$ still
containing $q_U^{-1}(q_U(\mathcal Y))$ and endowed with a finite
surjective morphism
$\Pi^*: \mathcal W\to\Aff^1\times U$ {\rm(}in general
$\neq\Pi${\rm)}.
\end{lem}
Let $\Pi:\mathcal X\to\Aff^1\times U$ be the above finite
surjective $U$-morphism.
%% from the hypotheses of Lemma \ref{Lemma_8_2}.
The following diagram summarises the situation:
$$ \xymatrix{
{}&\mathcal Z \ar[d]&{}\\
{\mathcal X - \mathcal Z\ }\ar@{^{(}->}@<-2pt>[r]&\mathcal X\ar@<2pt>[d]^{q_U}\ar[r]^>>>>{\Pi}& \Aff^1 \times U\\
{}& U\ar@<2pt>[u]^{\Delta}&{} } $$
\noindent
Here $\mathcal Z$ is the closed sub-scheme defined by the equation
$f=0$. By assumption, $\mathcal Z$ is finite over $U$. Let
$\mathcal Y=\Pi^{-1}(\Pi(\mathcal Z \cup \Delta(U)))$. Since
$\mathcal Z$ and $\Delta(U)$ are both finite over $U$ and since
$\Pi$ is a finite morphism of $U$-schemes, $\mathcal Y$ is also
finite over $U$. Denote by $y_1,\dots,y_m$ its closed points and
let $S=\text{Spec}(\mathcal O_{\mathcal X,y_1,\dots,y_m})$. Set
$T=\Delta(U)\subseteq S$. Further, let $G_U=\Delta^*(G_{\mathcal
X})$ be as in the hypotheses of Theorem
\ref{ThEquatingGroups} and
let
$G_{\const}$
be the pull-back of
$G_U$ to $\mathcal X$. Finally,
let
$\varphi:G_{\const}|_T \to G_{\mathcal X}|_T$
be the canonical
isomorphism. Recall that by assumption $\mathcal X$ is $U$-smooth,
and thus $S$ is regular.
\par
By Proposition
\ref{PropEquatingGroups} there exists a finite
\'etale covering $\theta_0:\tilde S\to S$, a section
$\delta:T\to\tilde S$ of $\theta_0$ over $T$ and an isomorphism
$$ \Phi_0:\theta^*_0(G_{\const,S})\to\theta^*_0(G_{\mathcal X}|_S) $$
\noindent
such that $\delta^*\Phi_0=\varphi$.
{\bf Replacing $\tilde S$ with a connected component of $\tilde S$ which contains
$\delta(T)=\delta(\Delta(U))$
we may and will assume that $\tilde S$ is irreducible.}
We can extend these data to a
neighborhood $\mathcal V$ of $\{y_1,\dots,y_n\}$ and get the
diagram
\begin{equation}
\xymatrix{
     {}  &  \tilde S \ar[d]^{\theta_0} \ar@{^{(}->}@<-2pt>[r]  & \tilde {\mathcal V}  \ar[d]_{\theta} &\\
     T \ar@{^{(}->}@<-2pt>[r] \ar[ur]^{
\delta} & S \ar@{^{(}->}@<-2pt>[r]  &   \mathcal V \ar@{^{(}->}@<-2pt>[r]  &  \mathcal X &\\
    }
\end{equation}
\noindent
where $\pi:\tilde{\mathcal V}\to\mathcal V$ finite \'etale, and an
isomorphism $\Phi:\theta^*(G_{\const})\to\theta^*(G_{\mathcal
X})$.
\par
Since $T$ isomorphically projects onto $U$, it is still closed
viewed as a sub-scheme of $\mathcal V$. Note that since $\mathcal
Y$ is semi-local and $\mathcal V$ contains all of its closed
points, $\mathcal V$ contains $\Pi^{-1}(\Pi(\mathcal Y))=\mathcal
Y$. By Lemma \ref{Lemma_8_2} there exists an open subset $\mathcal
W\subseteq\mathcal V$ containing $\mathcal Y$ and endowed with a
finite surjective $U$-morphism $\Pi^*:\mathcal W\to\Aff^1\times
U$.
\par
Let $\mathcal X^{\prime}=\theta^{-1}(\mathcal W)$,
$f^{\prime}=\theta^{*}(f)$, $q^{\prime}_U=q_U\circ\theta$, and let
$\Delta^{\prime}:U\to\mathcal X^{\prime}$ be the section of
$q^{\prime}_U$ obtained as the composition of $\delta$ with
$\Delta$. We claim that the triple $(\mathcal
X^{\prime},f^{\prime},\Delta^{\prime})$ is a nice triple. Let us
verify this. Firstly, the structure morphism
$q^{\prime}_U:\mathcal X^{\prime} \to U$ coincides with the
composition
$$
\mathcal X^{\prime}\xra{\theta}
\mathcal W\hra\mathcal X\xra{q_U} U.
$$
\noindent
Thus, it is smooth. The element $f^{\prime}$ belongs to the ring
$\Gamma(\mathcal X^{\prime},\mathcal O_{\mathcal X^{\prime}})$,
the morphism $\Delta^{\prime}$ is a section of $q^{\prime}_U$.
Each component of each fibre of the morphism $q_U$ has dimension
one, the morphism $\mathcal X^{\prime}\xra{\theta}\mathcal
W\hra\mathcal X$ is \'{e}tale. Thus, each component of each fibre
of the morphism $q^{\prime}_U$ is also of dimension one. Since
$\{f=0\} \subset {\mathcal W}$ and $\theta: \mathcal X^{\prime}
\to {\mathcal W}$ is finite, $\{f^{\prime}=0\}$ is finite over
$\{f=0\}$ and hence  also over $U$. In other words, the $\mathcal
O$-module $\Gamma(\mathcal X^{\prime},\mathcal O_{\mathcal
X^{\prime}})/f^{\prime}\cdot\Gamma(\mathcal X^{\prime},\mathcal
O_{\mathcal X^{\prime}})$ is finite. The morphism $\theta:
\mathcal X^{\prime}\to\mathcal W$ is finite and surjective.
We have constructed above in
Lemma \ref{Lemma_8_2}
the finite surjective morphism
$\Pi^*:\mathcal W\to\Aff^1\times U$. It follows that
$\Pi^*\circ\theta:\mathcal X^{\prime}\to\Aff^1\times U$ is finite
and surjective.
\par
Clearly, the \'{e}tale morphism $\theta:\mathcal X^{\prime}
\to\mathcal X$ is a morphism of nice triples, with $g=1$.
\par
Denote the restriction of $\Phi$ to $\mathcal X^{\prime}$ simply
by $\Phi$. The equality $(\Delta^{\prime})^*{\Phi}=\id_{G_U}$
holds by the very construction of the isomorphism $\Phi$. Theorem
follows.

\section{A basic nice triple}
\label{BasicTriple}

With Propositions \ref{ArtinsNeighbor} and \ref{CartesianDiagram}
at our disposal we may form {\it a basic nice triple}, namely the
triple (\ref{FirstNiceTriple}) below. This is the main aim of the
present section.
%a very
%nice diagram similar to the famous diagrams of Quillen and
%Voevodsky.

Namely, fix a smooth geometrically irreducible affine $k$-scheme $X$,
a finite family of points $x_1,x_2, \dots , x_n$ on $X$, and a
non-zero function $\ttf\in k[X]$.
% and a closed subset $Z$ of $X$ given by the equation $f=0$.
We {\it always assume} that the set $\{x_1,x_2,\dots,x_n\}$ is
contained in the vanishing locus of the function $\ttf$.
\par
%Replacing $k$ by its algebraic closure in $k[X]$, we may assume
%that $X$ is a geometrically irreducible $k$-variety.
%It is easy to see that $X$ remains smooth over $k$.
By Proposition \ref{ArtinsNeighbor} there exist a Zariski open
neighborhood $X^0$ of the family $\{x_1,x_2,\dots,x_n\}$ and an
almost elementary fibration $p:X^0\to S$, where $S$ is an open subscheme
of the projective space $\Pro^{\mydim X-1}$, such that
$$ p|_{\{\ttf=0\}\cap X^0}:\{\ttf=0\}\cap X^0\to S $$
\noindent
is finite surjective. Let $s_i=p(x_i)\in S$, for each $1\le i\le
n$. Shrinking $S$, we may assume that $S$ is {\it affine \/} and
still contains the family $\{s_1,s_2,\dots,s_n\}$. Clearly, in
this case $p^{-1}(S)\subseteq X^0$ contains the family
$\{x_1,x_2,\dots,x_n\}$. We replace $X$ by $p^{-1}(S)$ and $\ttf$
by its restriction to this new $X$.
%(we continue to denote this restriction by $f$).
%Respectively still write $Z$ for the vanish locus of $f$ on that new $X$.
\par
In this way we get an almost elementary fibration $p:X\to S$ such that
$$ \{x_1,\dots,x_n\}\subset\{\ttf=0\}\subset X, $$
\noindent
$S$ is an open affine subscheme in the projective space
$\Pro^{\mydim X-1}$, and the restriction
$p|_{\{\ttf=0\}}:\{\ttf=0\}\to S$ of $p$ to the vanishing locus of $\ttf$
is a finite surjective morphism. In other words, $k[X]/(\ttf)$ is
finite as a $k[S]$-module.
\par
As an open affine subscheme of the projective space $\Pro^{\mydim
X-1}$ the scheme $S$ is regular. By Proposition~\ref{CartesianDiagram}
one can shrink $S$ in such a way that $S$
is still affine, contains the family $\{s_1,s_2,\dots,s_n\}$ and
there exists a finite surjective morphism
$$ \pi:X\to\Aff^1\times S $$
\noindent
such that $p=\pr_S\circ\pi$. Clearly, in this case
$p^{-1}(S)\subseteq X$ contains the family
$\{x_1,x_2,\dots,x_n\}$. We replace $X$ by $p^{-1}(S)$ and $\ttf$
by its restriction to this new $X$.
\par
In this way we get an almost elementary fibration $p:X\to S$ such that
$$ \{x_1,\dots,x_n\}\subset\{\ttf=0\}\subset X, $$
\noindent
$S$ is an open affine subscheme in the projective space
$\Pro^{\mydim X-1}$, and the restriction
$p|_{\{\ttf=0\}}:\{\ttf=0\}\to S$
is a finite surjective morphism. Eventually we conclude that there
exists a finite surjective morphism $\pi:X\to\Aff^1\times S $ such that $p=\pr_S\circ\pi$.
\par
Now, set $U:=\text{Spec}(\mathcal O_{X,\{x_1,x_2,\dots,x_n\}})$,
denote by $\can:U\hra X$ the canonical inclusion of schemes, and
let $p_U=p\circ\can:U\to S$. Further, we consider the fibre
product
$$
\mathcal X:=U\times_{S}X.
$$
 Then the canonical projections
$q_U:\mathcal X\to U$ and $q_X:\mathcal X\to X$ and the diagonal
morphism $\Delta:U\to\mathcal X$ can be included in the following
diagram
\begin{equation}
\label{SquareDiagram}
    \xymatrix{
     \mathcal X\ar[d]_{q_U}\ar[rr]^{q_X} &&  X   & \\
     U \ar[urr]_{\can} \ar@/_0.8pc/[u]_{\Delta} &\\
    }
\end{equation}
where
\begin{equation}
\label{DeltaQx} q_X\circ\Delta=\can
\end{equation}
and
\begin{equation}
\label{DeltaQu} q_U\circ\Delta=\id_U.
\end{equation}
Note that $q_U$ is {\it a smooth morphism with geometrically
irreducible fibres of dimension one}. Indeed, observe that $q_U$
is a base change via $p_U$ of the morphism $p$ which has the
desired properties.
Note that $\mathcal X$ is irreducible. Indeed,
$U$ is irreducible and the fibre of $q_U$ over the generic point of $U$ is irreducible.

Taking the base change via $p_U$ of the finite
surjective morphism $\pi:X\to\Aff^1\times S$, we get {\it a finite
surjective morphism
$$ \Pi:\mathcal X\to\Aff^1\times U $$
\noindent
such that\/} $q_U=\pr_U\circ\Pi$, where $pr_U:\Aff^1\times U\to U$ is the natural projection.

Set $f:=q_X^*(\ttf)$. The $\mathcal O_{X,\{x_1,x_2,\dots,x_n\}}$-module
$\Gamma(\mathcal X,\mathcal O_{\mathcal X})/f\cdot\Gamma(\mathcal X,\mathcal O_{\mathcal X})$
%and
%$\mathcal Z:= q_X^{-1}(Z)$
is finite, since the $k[S]$-module $k[X]/\ttf\cdot k[X]$ is finite.
%That is
%$\mathcal Z$
%is the vanishing locus of
%$\textit{f}$ .
%The locus $\mathcal Z$ {\it is finite and surjective over} $U$,
%since $Z$ is finite surjective over $S$.

Now the data
\begin{equation}
\label{FirstNiceTriple} (q_U:\mathcal X\to U,f,\Delta)
\end{equation}
form an example of a {\it nice triple\/} as in
Definition
\ref{DefnNiceTriple}. Moreover, we have

\begin{clm}
\label{DeltaIsWellDefined} The schemes $\Delta(U)$ and $\{f=0\}$
are both semi-local and the set of closed points of $\Delta(U)$ is
contained in the set of closed points of $\{f=0\}$.
\end{clm}
This holds since the set $\{x_1,x_2,\dots,x_n\}$ is contained in
the vanishing locus of the function $\ttf$.
%\begin{equation}
%\xymatrix{U_K\ar[r]^{\bar\pi}\ar[d]&U\ar[d]\\
%{\spe K}\ar[r]^\pi\ar[ur]^y\ar@/^0.8pc/[u]^{\tilde y}
%\ar@/_0.8pc/[u]_{\tilde x}&{\spe k,}\ar@/_0.8pc/[u]_x}
%\end{equation}

%%%%%%%%%%%%%%%%%%%%%%%%%%%%%%%%%%%%%%%%%%%%%%%%%%%%%%%%%%%%%%%%%%%%%%%%%%%%%%%%%%%%%%%%%%%
%%%%%%%%%%%%%%%%%%%%%%%%%%%%%%%%%%%%%%%%%%%%%%%%%%%%%%%%%%%%%%%%%%%%%%%%%%%%%%%%%%%%%%%%%%%

\section{Main construction}
\label{MainConstruction}
The main result of this Section is Corollary
\ref{Ptandht}.

Fix a $k$-smooth irreducible affine $k$-scheme $X$, a finite family of points
$x_1,x_2,\dots,x_n$ on $X$, and set $\mathcal O:=\mathcal
O_{X,\{x_1,x_2,\dots,x_n\}}$ and $U:= \text{Spec}(\mathcal O)$.
Let $A$ be the Noetherian $k$-algebra from Theorem \ref{MainThmGeometric} and $T=\text{Spec}(A)$.
Further, consider a simple simply connected $U$-group scheme $G$
and a principal $G$-bundle $P$ over $\mathcal O \otimes_k A$ which is trivial
over $K \otimes_k A$ for the field of fractions $K$ of $\mathcal O$.
We may and will
assume that for certain $\text{f} \in \mathcal O$ the principal
$G$-bundle $P$ is trivial over $\mathcal O_{\text{f}} \otimes_k A$.

%Let $G_0$ be a split simple simply connected algebraic
%group over $k$ of the same type as $G$, and let $H:=\underline {Aut}_k(G_0)$ be its algebraic automorphism group;
%it is a $k$-smooth
%algebraic $k$-group.
Shrinking
$X$ if necessary, we may secure the following properties
%(see Section
%\ref{BasicTriple}).
\par\smallskip
(i) The points $x_1,x_2,\dots,x_n$ are still in $X$ and $X$ is affine.
\par\smallskip
(ii) The group scheme $G$ is defined over $X$ and it is a simple
group scheme. We will often denote this $X$-group scheme by $G_X$
and write $G_U$ for the original $G$.
%\par\smallskip
%(ii') Let $\mathcal H:=\underline {Iso}_X(G_0 \times X, G)$ be the isomorphism scheme.
%It is a principal homogeneous $H$-bundle over $X$, and the morphism
%$G_0 \times \mathcal H \to G$ taking $(h, \phi)$ to $\phi(h)$ identifies $X$-group schemes $G_0 \times_H \mathcal H$ and $G$.
%There are
%a finite \'{e}tale Galois cover
%$\pi: \tilde X \to X$
%of $X$ with a Galois group $\Gamma$ and a section
%$s: \tilde X \to \mathcal H$
%of the structure morphism
%$\mathcal H \to X$.
%be an $X$-scheme morphism ({\it in other words $\mathcal H$ splits over} $\tilde X$).}
%there is given a finite \'{e}tale morphism
%$\pi: \tilde X \to X$
%and an $\tilde X$-group scheme isomorphism
%$\pi^*(G_X) \to \tilde X \times_{Spec (k)} G_0$, where $G_0$ is a split simple simply connected algebraic
%group over $k$ of the same type as $G$,
\par\smallskip
(iii)
The principal $G_U$-bundle $P$ is the restriction to $U \times_{\Spec(k)} T$ of a principal $G_X$-bundle $P_X$ over $X\times_{\Spec(k)} T$
%defined over $X$
and
$\ttf\in k[X]$.
We often will write $P_U$ for the original principal $G_U$-bundle $P$ over $U \times_{\Spec(k)} T$.
%The principal $G$-bundle $P$ is defined over $X$ and the function $\text{f}$ belongs to $k[X]$.
\par\smallskip
(iv) The restriction $P_{\text{f}}$ of the bundle $P_X$ to the
principal open subset $X_{\text{f}} \times_{\Spec(k)} T$ is trivial and $\text{f}$
vanishes at each $x_i$'s.

After substituting $k$ by its algebraic closure $\tilde k$ in $k[X]$, and $T$ by
$\tilde T=\Spec(\tilde k)\times_{\Spec(k)} T$,
we can assume that $X$ is a $\tilde k$-smooth geometrically irreducible affine $\tilde k$-scheme.  Note that
$U\times_{\Spec(\tilde k)}\tilde T\cong U\times_{\Spec(k)}T$ as $U$-schemes, and the same holds for $X$ instead of $U$.
To simplify the notation, we will
continue to denote this new $\tilde k$ by $k$ and $\tilde T$ by $T$.

\smallskip
In particular, we are given now the smooth geometrically irreducible affine
$k$-scheme $X$, the finite family of points $x_1,x_2,\dots,x_n$ on
$X$, and the non-zero function $\ttf\in k[X]$ vanishing at each
point $x_i$. Recall that starting from these data we
constructed at the very end
of Section \ref{BasicTriple} the nice triple
(\ref{FirstNiceTriple}) of the form
$(q_U:\mathcal X\to U,f,\Delta)$
with $\mathcal X=U \times_S X$.
We did that shrinking $X$ and securing properties (i) to (iv) at the same time.

Recall that $q_X: \mathcal X= U \times_S X \to X$
is the projection to $X$.
%Let $G_X$ be the simple simply connected $X$-group scheme, $P_X$ be
%the principal $G_X$-bundle over $X$.
%The restriction $P_{X, \ttf}$ of
%the bundle $P_X$ to the principal open subscheme $X_{\ttf}$ is
%trivial by Item (iv) above.
%Recall, that starting
%from these data at the very end of Section
%\ref{BasicTriple} we constructed the nice triple
%(\ref{FirstNiceTriple}).
Set
$$
G_{\mathcal X}:=(q_X)^*(G_X)\quad\mbox{and}\quad G_{const}:=(q_U)^*(G_U).
$$

%Set $G_{\mathcal X}:=(q_X)^*(G)$.
By Theorem \ref{ThEquatingGroups}
there exists a morphism of nice triples
$$ \theta: (q^{\prime}_U: \mathcal X^{\prime} \to U,f^{\prime},\Delta^{\prime})\to (q_U: \mathcal X \to U,f,\Delta) $$
\noindent
and an isomorphism
\begin{equation}
\label{KeyEquation} \Phi: \theta^*(G_{\const}) \to
\theta^*(G_{\mathcal X})=: G_{\mathcal X^{\prime}}
%=(q_X \circ \theta)^*(G)=(q^{\prime}_X)^*(G)=: G_{\mathcal X^{\prime}} \ (see (\ref{QprimeX}))
\end{equation}
of $\mathcal X^{\prime}$-group schemes such that
$(\Delta^{\prime})^*(\Phi)=\id_{G_U}$.
\par
Set
\begin{equation}
\label{QprimeX} q^{\prime}_X=q_X\circ\theta:\mathcal X^{\prime}\to X.
\end{equation}
%$f^{\prime}= (q^{\prime}_X)^*(\text{f})$
Recall that
\begin{equation}
\label{QprimeU} q^{\prime}_U=q_U\circ\theta:\mathcal X^{\prime}\to
U,
\end{equation}
since
$\theta$
is a morphism of nice triples.\\

Note that, since by Claim~\ref{DeltaIsWellDefined} $f$ vanishes on all closed points of $\Delta(U)$,
and $\theta$ is a morphism of nice triples, $f^{\prime}$ vanishes on all closed points of $\Delta^{\prime}(U)$ as well.
Therefore, the nice triple
$(q^{\prime}_U: \mathcal X^{\prime} \to U, f^{\prime}, \Delta^{\prime}: U \to \mathcal X^{\prime})$
is subject to Theorem~\ref{ElementaryNisSquare}.

By Theorem \ref{ElementaryNisSquare} there exists a finite
surjective morphism $\sigma:\mathcal X^{\prime}\to\Aff^1\times U$
of $U$-schemes satisfying (1) to (3) from that Theorem. In
particular, one has
$$ \sigma^{-1}\Big(\sigma\big(\{f^{\prime}=0\}\big)\Big)=
N(f^{\prime})=\{f^{\prime}=0\}\sqcup\{g_{f^{\prime},\sigma}=0\}$$
%$g^{\prime}_\sigma$
\noindent
with $N(f^{\prime})$ and $g_{f^{\prime},\sigma}$ defined in the item (2) of
Theorem \ref{ElementaryNisSquare}. Thus, replacing for brevity
$g_{f^{\prime},\sigma}$ by $g^{\prime}$, one gets the following
%$g^{\prime}_\sigma$
elementary distinguished square in the category of $U$-smooth schemes (see [Def. 2.1, Vo]):
\begin{equation}
\label{ElemNisSquareDiagram}
    \xymatrix{
(\mathcal X^{\prime})^{0}_{N(f^{\prime})}= (\mathcal
X^{\prime})^{0}_{f^{\prime}g^{\prime}}\ar[rr]^{\inc}
\ar[d]_{\sigma^0_{f^{\prime}g^{\prime}}} &&
(\mathcal X^{\prime})^0_{g^{\prime}}\ar[d]^{\sigma^0_{g^{\prime}}}  &\\
(\Aff^1\times U)_{N(f^{\prime})}\ar[rr]^{\inc}&&\Aff^1\times U &\\
}
\end{equation}

The base change of this square by means of the morphism
$U \times_{\Spec(k)} T \to U$
is an elementary distinguished square in the category of smooth
$U \times_{\Spec(k)} T$-schemes. Thus this new square can be used to build up principal
$G_U$-bundles over
$(\Aff^1\times U) \times_{\Spec(k)} T$
beginning with certain data over the three other corners. This is what we are
going to do below in this Section.

%By property (ii') above
%we are provided with the finite \'{e}tale Galois cover
%$\pi: \tilde X \to X$
%of $X$ with the Galois group $\Gamma$ and a section
%$s: \tilde X \to \mathcal H$
%of the structure morphism
%$\mathcal H \to X$.
%The isomorphism $\Phi$ of principal homogeneous $H$-bundles over
%$\overline {X\times X}$
%from Lemma \ref{equating2} induces an isomorphism
%$\Psi: \rho^*(G_1) \to \rho^*(G_2)$
%(see (\ref{equating3}))
%of
%$\overline {X\times X}$-group schemes.
%Restricting it to
%$\mathcal X^{\prime} \subset \overline{U\times_S X} \subset \overline {X\times X}$,
%we get an isomorphism
%\begin{equation}
%\label{KeyEquation} \Psi_S=\Psi|_{\mathcal X^{\prime}}: \rho_S^*(G_{\const}) \to
%\rho_S^*(G_{\mathcal X})
%=: G_{\mathcal X^{\prime}}
%=(q_X \circ \theta)^*(G)=(q^{\prime}_X)^*(G)=: G_{\mathcal X^{\prime}} \ (see (\ref{QprimeX}))
%\end{equation}
%of $\mathcal X^{\prime}$-group schemes such that
%$(\Delta^{\prime})^*(\Psi_S)=\id_{G_U}$.

%Set $q^{\prime}_X=q_X\circ\rho_S:\mathcal X^{\prime}\to X$
%\begin{equation}
%\label{QprimeX} q^{\prime}_X=q_X\circ\rho_S:\mathcal X^{\prime}\to X.
%\end{equation}
%$f^{\prime}= (q^{\prime}_X)^*(\text{f})$
%and recall that $q^{\prime}_U=q_U\circ\rho_S:\mathcal X^{\prime}\to U$.
Set
$$Q^{\prime}_X=q^{\prime}_X \times id_T: \mathcal X^{\prime} \times_{\Spec(k)} T \to X \times_{\Spec(k)} T,$$
$$Q^{\prime}_U=q^{\prime}_U \times id_T: \mathcal X^{\prime} \times_{\Spec(k)} T \to U \times_{\Spec(k)} T.$$
%\begin{equation}
%\label{QprimeU} q^{\prime}_U=q_U\circ\rho_S:\mathcal X^{\prime}\to
%U,
%\end{equation}
%since
%$\rho_S$
%is a morphism of nice triples.
Consider $(Q^{\prime}_X)^*(P_X)$ as a principal
$(q^{\prime}_U)^*(G_U)=\theta^*(G_{\const})$-bundle via the isomorphism
$\Phi$. Recall that $P_X$ is trivial as a principal $G_X$-bundle over
$X_{\text{f}} \times_{\Spec(k)} T$. Therefore,
$(Q^{\prime}_X)^*(P_X)$ is trivial as a
principal
$\rho_S^*(G_{\mathcal X})$-bundle
%$G_{\mathcal X^{\prime}}$-bundle
over
$\mathcal X^{\prime}_{f^{\prime}} \times_{\Spec(k)} T$.
%Clearly, the vanishing locus of the function
%$(q^{\prime}_X)^*(\text{f})$
%coincides with
%the closed subscheme
%$\theta^{-1}(\mathcal Z)$
%of the scheme
%$\mathcal X^{\prime}$.
%Since $\rho_S$ is a nice triple morphism one has
%$\theta^{-1}(\mathcal Z) \subset \mathcal Z^{\prime}$,
%$f^{\prime}=\theta^*(f)\cdot h^{\prime}$, and thus the principal
%$G_{\mathcal X^{\prime}}$-bundle
%$(q^{\prime}_X)^*(P_X)=(q_X\circ\theta)^*(P_X)$ is trivial over
%$\mathcal X^{\prime}_{f^{\prime}}$.
%\par
So, $(Q^{\prime}_X)^*(P_X)$ is trivial over
$\mathcal X^{\prime}_{f^{\prime}} \times_{\Spec(k)} T$, when regarded as a principal
$\rho_S^*(G_{{\const}})$-bundle via the isomorphism $\Psi_S$.
%{\bf The restriction $P_{X, \ttf}$ of
%the bundle $P_X$ to the principal open subscheme $X_{\ttf} \times_{Spec(k)} T$ is
%trivial by Item (iv) above.}
\par

%Recall that $\rho_S^*(G_{\const})=\rho_S^*(q_U^*(G_U))=(q_U^{\prime})^*(G_U)$.
Thus, regarded as a principal $G_U$-bundle,
the bundle $(Q^{\prime}_X)^*(P_X)$ over $\mathcal X^{\prime} \times_{\Spec(k)} T$
becomes trivial over $\mathcal X^{\prime}_{f^{\prime}} \times_{\Spec(k)} T$, and a
fortiori over $(\mathcal X^{\prime})^{0}_{f^{\prime}g^{\prime}} \times_{\Spec(k)} T$.
Now, taking the trivial $G_U$-bundle over $(\Aff^1\times
U)_{N(f^{\prime})}$ and an isomorphism
\begin{equation}
\label{Skleika}
\psi: G_U \times_U [(\mathcal X^{\prime})^{0}_{N(f^{\prime})} \times_{\Spec(k)} T]  \to
(Q^{\prime}_X)^*(P_X)|_{[(\mathcal X^{\prime})^{0}_{N(f^{\prime})}\times_{\Spec(k)} T] }
\end{equation}
of principal $G_U$-bundles, we get {\it a principal $G_U$-bundle $\mathcal G_t$
over} $(\Aff^1\times U) \times_{\Spec(k)} T$ such that
\begin{itemize}
\item[(1)]
$\mathcal G_t|_{[(\Aff^1\times U)_{N(f^{\prime})}\times_{\Spec(k)} T]}=G_U \times_U [(\Aff^1\times U)_{N(f^{\prime})}\times_{\Spec(k)} T]$
%it is trivial over $(\Aff^1\times U)_{N(f^{\prime})}$,
%%
\item[(2)]
there is an isomorphism
$\varphi: [(\sigma^0_{g^{\prime}}) \times \id_T]^*(\mathcal G_t) \to (Q^{\prime}_X)^*(P_X)|_{[(\mathcal X^{\prime})^{0}_{g^{\prime}}\times_{\Spec(k)} T ]}$
of the principal $G_U$-bundles,
%$(\sigma)^*(P_t)|_{(\mathcal X^{\prime})^{0}_{g^{\prime}}}$ and $(q^{\prime}_X)^*(P_X)|_{(\mathcal X^{\prime})^{0}_{g^{\prime}}}$ are
%isomorphic as principal $G_U$-bundles.
where $(Q^{\prime}_X)^*(P_X)$
is regarded as a principal $G_U$-bundle via the $\mathcal X^{\prime}$-group scheme isomorphism $\Phi$ from
(\ref{KeyEquation}); %%
\item[(3)]
$(\inc \times \id_T)^*(\varphi)=\psi$.
%|_{(\mathcal X^{\prime})^{0}_{N(f^{\prime})}}
%over $(\mathcal X^{\prime})^{0}_{N(f^{\prime})}$ the
%two $G_U$-bundles are identified via the isomorphism $\psi$ from
%(\ref{Skleika}).
\end{itemize}

%Note that
%$\sigma (\mathcal Z^{\prime})$
%is a divisor on
%$\Aff^1 \times U$. In fact,
%$\mathcal Z^{\prime}$
%is a divisor on
%%$\mathcal X^{\prime}$
%by the item $(b)$ of Definition
%\ref{DefnNiceTriple} and the morphism
%$\sigma$ is finite surjective.
%Clearly,
%$\sigma (\mathcal Z^{\prime})$
%is finite surjective over $U$.
%The scheme $U$ is semi-local regular. Thus the ring
%$\mathcal O_{X, \{x_1,x_2,\dots,x_n\}}[t]$
%is factorial.So, there exists an element
%$h(t) \in \mathcal O_{X, \{x_1,x_2,\dots,x_n\}}[t]$
%with the vanishing locus equal to
%$\sigma (\mathcal Z^{\prime})$.
%The polynomial $h(t)$ is unitary, since
%the divisor
%$\sigma (\mathcal Z^{\prime})$
%is finite surjective over $U$.
%{\bf Whence,
%$\sigma (\mathcal Z^{\prime})= \{h(t)=0\}$,
%where $h(t) \in \mathcal O_{X, \{x_1,x_2,\dots,x_n\}}[t]$
%is a unitary polynomial}.

Finally, form the following diagram
\begin{equation}
\label{DeformationDiagram2}
    \xymatrix{
(\Aff^1 \times U)\times_{\Spec(k)} T\ar[drr]_{\pr_U\times \id}&&(\mathcal
X^{\prime})^0_{g^{\prime}}\times_{\Spec(k)} T \ar[d]^{}
\ar[ll]_{\sigma^0_{g^{\prime}}\times \id}\ar[d]_{q^{\prime}_U \times \id}
\ar[rr]^{Q^{\prime}_X=q^{\prime}_X \times \id}&&X \times _{\Spec(k)} T&\\
&&U\times_{\Spec(k)} T \ar[urr]_{\can \times \id}\ar@/_0.8pc/[u]_{\Delta^{\prime}\times \id} &\\
    }
\end{equation}
This diagram is well-defined, since by Item (4) of Theorem~\ref{ElementaryNisSquare}
the image of the morphism
$\Delta^{\prime}$ lands in $(\mathcal X^{\prime})^0_{g^{\prime}}$.
%It follows from Item (4) of Theorem \ref{ElementaryNisSquare},
%that this diagram is well-defined.

%We need to check that
%$\Delta^{\prime} : U \to \mathcal X^{\prime}_0$
%lands into
%$(\mathcal X^{\prime})^{0}_{g^{\prime}}$.
%This follows easy from three facts:
%\begin{itemize}
%\item
%schemes
%$\Delta^{\prime}(U)$,
%$\{f^{\prime}=0\}$
%%$\mathcal Z^{\prime}$
%and
%$\{g^{\prime}=0\}$
%$\mathcal Z^{\prime}_1$
%are semi-local; $\Delta^{\prime}(U) \cup \{f^{\prime}=0\} \subset (\mathcal X^{\prime})^0$,
%\item
%the set of closed points of
%$\Delta^{\prime}(U)$
%is contained in the set of closed points of
%$\{f^{\prime}=0\}$ (see Remark
%\ref{DeltaIsWellDefined}),
%%$\mathcal Z^{\prime}$,
%\item
%$\{f^{\prime}=0\} \cap \{g^{\prime}=0 \} = \emptyset$.
%\end{itemize}

\begin{thm}
\label{MainData}
The principal $G_{U}$-bundle $\mathcal G_t$ over $(\Aff^1 \times U)\times _{\Spec(k)} T$,
the monic polynomial $N(f^{\prime}) \in \mathcal O[t]$, the diagram~{\rm(\ref{DeformationDiagram2})},
and the isomorphism $\Phi$ from~{\rm(\ref{KeyEquation})} constructed above,
%%{\rm\ref{MainConstruction}}
satisfy the following conditions {\rm(1*)}--{\rm(6*)}.
%from Section \ref{Outline}.
\par\smallskip
(1*) $q^{\prime}_U=\pr_U\circ\sigma^0_{g^{\prime}}$,
\par\smallskip
(2*) $\sigma^0_{g^{\prime}}$ is \'etale,
\par\smallskip
(3*) $q^{\prime}_U\circ\Delta^{\prime}=\id_U$,
\par\smallskip
(4*) $q^{\prime}_X\circ\Delta^{\prime}=\can$,
\par\smallskip
(5*) the restriction of $\mathcal G_t$ to $(\Aff^1 \times U)_{N(f^{\prime})} \times _{\Spec(k)} T$ is a trivial
$G_U$-bundle,
%% where $G_U$ is from the item $(ii)$ above;
\par\smallskip
(6*) $(\sigma^0_{g^{\prime}} \times \id)^*(\mathcal G_t)$ and $(Q^{\prime}_X)^*(P_X)$ are isomorphic as
$G_U$-bundles over $(\mathcal X^{\prime})^0_{g^{\prime}}\times_{\Spec(k)} T$. Here $(Q^{\prime}_X)^*(P_X)$ is regarded as a principal
$G_U$-bundle via the group scheme isomorphism $\Phi$ from {\rm(\ref{KeyEquation})}.
% and the identity $(q^{\prime}_U)^*(G_U)=\rho_S^*(G_{\const})$
\par\smallskip
\end{thm}

\begin{proof} By the very choice of $\sigma$ it is an $U$-scheme
morphism, which proves (1*). By the choice of $(\mathcal X^{\prime})^0\hra\mathcal X^{\prime}$
in Theorem~\ref{ElementaryNisSquare},
the morphism $\sigma$ is \'etale on this subscheme, hence one gets (2*).
Property (3*) holds for $\Delta^{\prime}$ since
$(q^{\prime}_X: \mathcal X^{\prime} \to U, f^{\prime},\Delta^{\prime})$
is a nice triple
and, in particular,
$\Delta^{\prime}$ is a section of
$q^{\prime}_U$. Property (4*) can be established as follows:
$$ q^{\prime}_X\circ\Delta^{\prime}=
(q_X\circ\rho_S)\circ\Delta^{\prime}=q_X\circ\Delta=\can. $$
\noindent
The first equality here holds by the definition of $q^{\prime}_X$,
%see (\ref{QprimeX});
the second one holds since $\rho_S$ is a
morphism of nice triples; the third one follows from equality
(\ref{DeltaQx}). Property (5*) is just Property (1) in the above
construction of $\mathcal G_t$.
%recall that the principal
%$G_U$-bundle $P_t$ is trivial over
%$(\Aff^1 \times U)_{N(f^{\prime})}$
%by the very construction of $P_t$, which is done just above. Finally,
Property (6*) is precisely Property (2) in the construction of
$\mathcal G_t$.
\end{proof}
The composition
$$
s^{\prime}:= \sigma^0_{g^{\prime}} \circ \Delta^{\prime}:U\to\Aff^1\times U
$$ is a section of the projection $\pr_U$
by the properties (1*) and (3*). Recall that $G_U$ over $U$ is the original group scheme $G$ introduced in the very
beginning of this Section.
Since $U$ is semi-local, we may assume that $s^{\prime}$ {\it is the zero section of the projection} $\Aff^1_U \to U$.
Furthermore, making an affine transformation of $\Aff^1_U \to U$, we may assume that
$N(f^{\prime})(1) \in \mathcal O$ {\it is invertible}.
\begin{cor}[{\bf=Theorem \ref{MainHomotopy}}]
\label{Ptandht}
The principal $G_U$-bundle $\mathcal G_t$ over $\Aff^1_{[U\times_{\Spec(k)} T]}$ and the monic polynomial
$N(f^{\prime}) \in \mathcal O[t]$
%and the section $s$
%of the projection $\text{pr}: \Aff^1_U \to U$
are subject to the following conditions
\par\smallskip
(i) the restriction of $\mathcal G_t$ to $[(\Aff^1\times U)_{N(f^{\prime})}\times_{\Spec(k)} T]$ is a trivial
%$(\Aff^1_U)_{N(f^{\prime})}$ is a trivial
$G_U$-bundle,
\par\smallskip
(ii) the restriction of $\mathcal G_t$ to $\{0\} \times U \times_{\Spec(k)} T$ is the original $G_U$-bundle $P_U$.
\par\smallskip
(iii) $N(f^{\prime})(1) \in \mathcal O$ is invertible.
\end{cor}

\begin{proof}
The property (i) is just the property (5*) above. Now by (6*) the $G_U$-bundles
$$\mathcal G_t|_{\{0\} \times U \times_{\Spec(k)} T}=(s^{\prime}\times \id)^*(\mathcal G_t)=
{(\Delta^{\prime}\times \id)}^*((\sigma^0_{g^{\prime}}\times \id)^*(\mathcal G_t)) \ \text{and}$$
$$ \ {(\Delta^{\prime}\times \id)}^*(Q^{\prime}_X)^*(P_X)=(\can\times \id)^*(P_X)$$
are isomorphic, since ${\Delta^{\prime}}^*(\Phi)=\id_{G_U}$.
It remains to recall that the principal $G_U$-bundle $(\can\times \id)^*(P_X)$ is the original $G_U$-bundle $P_U$
by the choice of $P_X$. Whence the Corollary.

\end{proof}

%\begin{thm}
%\label{RelativeSpecialization}
%Let
%$(\mathcal X, \mathcal Z, \Delta)$
%be a nice triple over the semi-local scheme
%$U$ from Definition
%\ref{DefnNiceTriple}. Let
%$G$ be a simple adjoint
%$\mathcal X$-group scheme containing a split torus $\Bbb G_m$.
%Let $P$ be a principal $G$-bundle over
%$\mathcal X$
%such that
%$P|_{\mathcal X - \mathcal Z}$
%is a trivial principal
%$G$-bundle. Then there exists a morphism
%$\theta: (\mathcal X^{\prime}, \mathcal Z^{\prime}, \Delta^{\prime}) \to
%(\mathcal X, \mathcal Z, \Delta)$
%of triples such that the principal
%$(\theta^*)(G)$-bundle is trivial.
%
%In particular, the principal
%$\Delta^*(G)$-bundle
%$\Delta^*(P)$
%is trivial.
%\end{thm}

\section{Group of points of an isotropic simple group}

In this section we establish several results concerning groups of
points of simple groups, in particular, \label{GroupPart} Lemma
\ref{Surjectivity}, Proposition \ref{Key} and Lemma
\ref{AlphaUtimesbeta}, which play crucial role in the rest of the
paper.

\begin{defn}\label{E_P}
Let $G$ be a
reductive group scheme over a commutative ring $A$. Assume that $G$ has a proper
parabolic subgroup $P=P^+$ over $A$, and denote by $U^+$ its
unipotent radical.
%By \cite[Exp.~XXVI Cor.~2.3]{SGA}, there exists a Levi subgroup $L$ of $P$, and by
%\cite[Exp. XXVI Th. 4.3.2]{SGA} there exists a unique parabolic subgroup $P^-$ opposite to $P^+$,
By \cite[Exp.~XXVI Cor.~2.3, Th. 4.3.2]{SGA3} there exists a
parabolic subgroup $P^-$ of $G$ opposite to $P^+$, and
by~\cite[Exp.~XXVI Cor.~1.8]{SGA3} any two such subgroups are
conjugate by an element of $U^+(A)$. Let $U^-$ be the unipotent
radical of $P^-$. For any commutative $A$-algebra $B$ we define the
$P$-elementary subgroup $E_P(B)$ of the group $G(B)$ as follows:
$$ E_P(B)=\langle U^+(B),U^-(B)\rangle. $$
\end{defn}
\begin{lem}
\label{Surjectivity} Let $B\to\bar B$ be a surjective $A$-algebra
homomorphism. Then the induced homomorphism of elementary groups
$E_P(B)\to E_P(\bar B)$ is also surjective.
%% $$ \langle U^{+}_l(A), U^{-}_l(A)\rangle \ \to \
%% \langle U^{+}_l(\overline A), U^{-}_l(\overline A)\rangle $$
%% \noindent
\end{lem}
\begin{proof}
By~\cite[Exp.XXVI Cor. 2.5]{SGA3} the $A$-schemes $U^+$ and $U^-$
are isomorphic to $A$-vector bundles of finite rank. Thus, the
maps $U^{\pm}(B)\to U^{\pm}(\bar B)$ are surjective.
\end{proof}

Let $l$ be a field and $G_l$ be an isotropic simple
simply connected $l$-group scheme. Recall that an isotropic scheme
contains an $l$-split rank one torus $\Bbb G_{m,l}$. Choose and
fix two opposite parabolic subgroups $P_l=P^+_l$ and $P^-_l$ of
the $l$-group scheme $G_l$. Let $U^+_l$ and $U^-_l$ be their
unipotent radicals. We will be interested mostly in the group of
points $G_l(l(t))$. The following definition originates
from~\cite[Main theorem]{T}.
\begin{defn}
\label{G+} Define $G_l(l(t))^+$ as the subgroup of the group
$G_l(l(t))$ generated by $l(t)$-points of unipotent radicals of
all parabolic subgroups of $G_l$ defined over the field $l$.
\end{defn}

\begin{rem}
\label{AppearanceT-} Clearly, $l(t)=l(t^{-1})$. Thus,
$$ G_l(l(t))=G_l(l(t^{-1}))\qquad \text{and}\qquad
G_l(l(t))^{+}=G_l(l(t^{-1}))^{+}. $$
\end{rem}

By definition the group $G_l(l(t))^+$ is generated by unipotent
radicals of {\it all\/} $l$-parabolic subgroups, and thus contains
the elementary group $E_{P_l}(l(t))$, introduced in
Definition~\ref{E_P}. In fact they coincide.
\begin{prop}
\label{G+AndU+U-} The group $G_l(l(t))^+$ is generated by
$l(t)$-points of unipotent radicals of any two opposite parabolic
subgroups of the $l$-group scheme $G_{l}$. In particular, one has
the equality
\begin{equation}
\label{U+U-} G_l(l(t^{-1}))^{+}=\Big\langle U^{+}_l(l(t^{-1})),
U^{-}_l(l(t^{-1}))\Big\rangle=E_{P_l}(l(t^{-1})).
\end{equation}
\end{prop}
\begin{proof}
Set $G_{l(t)}=G_l\times_{\Spec\,l}\Spec\,l(t)$. The group
$G_l(l(t))^{+}$ is contained in the subgroup of
$G_l(l(t))=G_{l(t)}(l(t))$ generated by $l(t)$-points of unipotent
radicals of all parabolic subgroups of the group scheme $G_{l(t)}$
defined over the field $l(t)$. By~\cite[Prop.6.2.(v)]{BT73} the
latter group is generated by $l(t)$-points of unipotent radicals
of any two opposite parabolic subgroups of $G_{l(t)}$, in
particular, by $l(t)$-points of $U_{l(t)}^+$ and $U^-_{l(t)}$.
Since $U^\pm_{l(t)}(l(t))=U^\pm_{l}(l(t))$, we have~\eqref{U+U-}.
\end{proof}

\begin{rem}
\label{ConvinienceOfU+U-} For any commutative ring $A$ and an $A$-algebra $B$,
and any reductive $A$-group scheme $G$ we can define the group
$G_A(B)^+$ as in the Definition~\ref{G+}, that is, as the subgroup
generated by $B$-points of unipotent radicals of all $A$-parabolic
subgroups of $G$. The question, whether this subgroup coincides
with $E_P(B)$ for an $A$-parabolic subgroup $P$ of $G$, is in
general rather subtle. See the paper {\rm \cite{PSt}} by V.~Petrov
and the second author for details.
\end{rem}

%\begin{rem}
%\label{ConvinienceOfU+U-} For each field extension $L/l$ one can
%define the group $G_l(L)^{+}$ literally repeating the Definition
%\ref{G+}. However, the {\it $P_l$-elementary subgroup\/}
%$$ E_{P_l}(A)=\langle U^{+}_l(A),U^{-}_l(A)\rangle $$
%\noindent
%of the group $G_l(A)$ is well-defined for an arbitrary $l$-algebra
%$A$. Furthermore, it is functorial: an $l$-algebra homomorphism
%$A\to A^{\prime}$ induces a group homomorphism
%$$ E_{P_l}(A)\to E_{P_l}(A'). $$
%%% $$ \langle U^{+}_l(A), U^{-}_l(A)\rangle \ \to \ \langle
%%% U^{+}_l(A^{\prime}), U^{-}_l(A^{\prime})\rangle. $$
%\noindent
%Functoriality and other nice properties of the elementary group
%%% $\langle U^{+}_l(A), U^{-}_l(A)\rangle$
%make it more useful. This will be amply illustrated below.
%\par
%As another side remark let us observe that the main result of the
%paper {\rm \cite[]{PS}} by V.~Petrov and the third author asserts
%that for simple groups of relative rank at least $2$ {\rm(}and in
%fact under somewhat weaker assumptions{\rm)} the elementary group
%$E_{P_l}(A)$ does not depend on the choice of an $l$-parabolic.
%Recall, that the groups of relative rank at least $2$ countain an
%$l$-split rank two torus $\Bbb G^2_{m,l}$.
%\end{rem}

The following result is crucial for the sequel.
\begin{prop}
\label{Key}
One has the equality
\begin{equation}
\label{RaghunathanBT} G_l(l(t^{-1}))=G_l(l(t^{-1}))^{+}\cdot
G_l(l),
\end{equation}
where $G_l(l(t^{-1}))^{+}$ is the group defined in Definition
\ref{G+} {\rm(}see also Remark \ref{AppearanceT-}{\rm)}.
\end{prop}
\begin{proof}
This is proved in~\cite[Th\'eor\`eme 5.8]{Gille07}.
\end{proof}

%% Now, let $l[t]^{\bullet}$ be the set of all non-zero polynomials
%% in one variable $t$.
Let $f(t)\in l[t]$ be a polynomial of degree $n=\deg(f)$ in $t$
such that $f(0)\neq 0$. We consider the reciprocal polynomial
$$ f^*(t^{-1}):=f(t)/t^n\in l[t^{-1}];$$
\noindent
clearly, $f^*(0)\neq 0$.
Conversely, if we are given a polynomial $g(t^{-1})\in l[t^{-1}]$
of degree $n=\deg(g)$ in $t^{-1}$ such that
$g(0)\neq 0$, we define the reciprocal polynomial
$$ g^*(t):=g(t^{-1})\cdot t^n\in l[t]. $$
\noindent
%is defined by a similar formula.
%% is such that $g^*(0)\neq 0$.
The above correspondences are mutually inverse. Further, when
$f(t)\in l[t]$ runs over all polynomials in $t$ with $f(0)\neq 0$,
then the reciprocal polynomial $f^*(t^{-1})$ runs over all
polynomials $g(t^{-1})\in l[t^{-1}]$ with $g(0)\neq 0$.
\par
Now let us return to the setting considered in (\ref{U+U-}) and
(\ref{RaghunathanBT}) and Remark~\ref{AppearanceT-}. Each
non-constant $f(t)\in l[t]$ admits a unique factorisation of the
form $f(t)=t^r\cdot g(t)$, where $g(t)\in l[t]$ and $g(0)\neq 0$.
Clearly, for each $h(t)\in l[t]$ with $h(0)\neq 0$ one gets the
following inclusions
$$ G_l\big(l[t]_{f}\big)\le G_l\big(l[t^{-1},t]_{g^*}\big)\le
G_l\big(l[t^{-1},t]_{g^*h^*}\big). $$
\noindent
This leads us to the following Lemma.

\begin{lem}\label{AlphaUtimesbeta}
Let $P_l$ be an arbitrary parabolic $l$-subgroup of $G_l$.
For each $\alpha\in G_l(l[t]_{f(t)})$ one
can find a polynomial $h(t)\in l[t]$, $h(0)\neq 0$, and elements
$$ u\in E_{P_l}\big(l[t^{-1},t]_{g^*h^*}\big),\qquad
%% \langle U^+_l(l[t^{-1},t]_{g^*h^*}), \
%% U^+_l(l[t^{-1},t]_{g^*h^*})\rangle
 \beta\in G_l\big(l\big) $$
\noindent
such that
%$\alpha \in G(l[t^{-1},t]_{g^*h^*})$ and
\begin{equation}
\label{AlphaEqUtimesBeta} \alpha = u\beta\in
G_l\big(l[t^{-1},t]_{g^*h^*}\big).
\end{equation}
\noindent
The chain of $l$-algebra inclusions
$l[t]_{fh}\subseteq l[t]_{tfh}=l[t,t^{-1}]_{gh}=
l[t^{-1},t]_{g^*h^*}$
shows that
%$$u \in E_{P_l}(l[t]_{tfh})$
%and
$$ u\in E_{P_l}\big(l[t]_{tfh}\big),\quad\text{and}\quad
%% \langle U^+_l(l[t]_{tfh}),U^+_l(l[t]_{tfh})\rangle,
\alpha\in G_l\big(l[t]_{tfh}\big). $$
\end{lem}
\begin{proof}
As observed above, inclusions $\alpha\in
G_l\big(l[t^{-1},t]_{g^*}\big)\le
G_l\big(l[t^{-1},t]_{g^*h^*}\big)$ are obvious.
The equalities
(\ref{RaghunathanBT}) and (\ref{U+U-}) imply that there exists a
polynomial $h(t)\in l[t]$, $h(0)\neq 0$, and elements
$$ u\in E_{P_l}\big([t^{-1},t]_{g^*h^*}\big),\qquad
%% \langle U^+_l(l[t^{-1},t]_{g^*h^*},
%% U^-_l(l[t^{-1},t]_{g^*h^*})\rangle ,\
\beta\in G_l\big(l\big), $$
\noindent
such that $\alpha=u\beta$ in $G_l\big(l[t^{-1},t]_{g^*h^*}\big)$.
The last assertion of the Lemma follows from the obvious $l$-algebra inclusions
$$
l[t]_{fh}\subseteq l[t]_{tfh}=l[t,t^{-1}]_{gh}=
l[t^{-1},t]_{g^*h^*}.
$$
\end{proof}

%%%%%%%%%%%%%%%%%%%%%%%%%%%%%%%%%%%%%%%%%%%%%%%%%%%%%%%%%%%%%%%%%

\section{Principal $G$-bundles on a projective line}
\label{SecProjectiveLine}

The main result of the present section is
Corollary~\ref{TrivialOnFibraIsTrivialGlobal}, which implies Theorem~\ref{MainThmArithmetic}.
%Theorem \ref{NewBundle}.

Let $B$ be a Noetherian commutative ring, and let $\Aff^1_B$ and $\Pro^1_B$
be the affine line and the projective line over $B$, respectively.
Usually we identify the affine line with a subscheme of the
projective line as follows
$\Aff^1_B=\Pro^1_B-(\{\infty\}\times\text{Spec}(B))$, where
$\infty=[0:1]\in\Pro^1$. Let $G$ be a semi-simple $B$-group
scheme, let $P$ a principal $G$-bundle over $\Aff^1_B$, and let
$p:P\to\Aff^1_B$ be the corresponding canonical projection.

For a monic polynomial
$$ f=f(t)=t^n+a_{n-1}t^{n-1}+\dots+a_0\in B[t] $$
\noindent
we set $P_f=p^{-1}((\Aff^1_B)_f)$. Clearly, it is a principal
$G$-bundle over $(\Aff^1_B)_f$. Further, we denote by
$$ F(t_0,t_1)=t_1^n+a_{n-1}t_1^{n-1}t_0+\dots+a_0t_0^n $$
\noindent
the corresponding homogeneous polynomial in two variables. Note
that the intersection of the principal open set in $\Pro^1_B$
defined by the inequality $F\neq 0$ with the affine line
$\Aff^1_B$ equals the principal open subset $(\Aff^1_B)_f$. As in
the previous section in the case where $a_0\neq 0$ we consider the
reciprocal polynomial
$f^{*}(t^{-1})\in B[t^{-1}]$ equal to $f(t)/t^n$. %% $f^*$
\begin{defn}
\label{PofPhiAndf} Let $\varphi: G_{(\Aff^1_B)_f} \to P_f$ be a
principal $G$-bundle isomorphism. We write $P(\varphi, f)$ for a
principal $G$-bundle over the projective line $\Pro^1_B$ obtained
by gluing $P$ and $G_{(\Pro^1_B)_F}$ over $(\Aff^1_B)_f$ via the
principal $G$-bundle isomorphism $\varphi$.
\end{defn}
\begin{rem}\label{PropertiesPofPhiAndf}
For any $\varphi$ and $f$ as above, the
principal $G$-bundles $P(\varphi,f)$ and $P(\varphi,fg)$ coincide
for each monic polynomial $g\in B[t]$.
\par
For any $\varphi$ and $f$, any monic polynomial $h(t)\in B[t]$
such that $h(0)\in B^{\times}$ is invertible, and any $\beta\in
G(B[t^{-1}]_{h^*})$ the principal $G$-bundles
$P(\varphi,f)$ and $P(\varphi\circ\beta,tfh)$ are isomorphic. In fact,
they differ by a co-boundary.
\end{rem}

\begin{lem}\label{RR}
Let $l$ be a field and $G_l$ be a semi-simple
$l$-group scheme. Let $f\in l[t]$ be a non-constant polynomial. Let $P$ be a
principal $G_l$-bundle over $\Aff^1_l$ such that $P_f$ is trivial
over $\Aff^1_f$. Let $\varphi:G_{\Aff^1_{f}}\to P_f$ be a
principal $G_l$-bundle isomorphism. Let $P(\varphi,f)$ be the
corresponding principal $G$-bundle over $\Pro^1_l$.
Then there exists an $\alpha\in
G(l[t]_f)$ such that the principal $G$-bundle
$P(\varphi\circ\alpha,f)$ is trivial over $\Pro^1_l$.
\end{lem}
\begin{proof}
By~\cite[Prop. 2.2]{C-TO}
one has
%Let $G^{\prime}$ be a reductive $K$-group scheme from this Section
%(it is connected and even geometrically connected, since we follow
%\cite[Exp.~XIX, Defn.~2.7]{D-G}). Then
$$
\ker[\textrm{H}^1_{\text{\'et}}(l[t],G_l) \to \textrm{H}^1_{\text{\'et}}(l(t),G_l)]=* \ .
$$
%By the main theorem of \cite{RR}
So, we may assume that there is an
isomorphism $G_{\Aff^1_l}=P$ over $\Aff^1_l$. In this case the
above isomorphism $\varphi$ coincides with the right
multiplication by an element $\beta\in G_l(l[t]_f)$. Clearly,
$P(\beta\circ\beta^{-1},f)$ is trivial over $\Pro^1_l$. Thus,
$P(\varphi\circ\alpha,f)$ is trivial for $\alpha=\beta^{-1}$.
%By the main theorem of \cite{RR} we may assume that there is an
%isomorphism $G_{\Aff^1_l}=P$ over $\Aff^1_l$. In this case the
%above isomorphism $\varphi$ coincides with the right
%multiplication by an element $\beta\in G_l(l[t]_f)$. Clearly,
%$P(\beta\circ\beta^{-1},f)$ is trivial over $\Pro^1_l$. Thus,
%$P(\beta\circ\alpha,f)$ is trivial for $\alpha=\beta^{-1}$.
\end{proof}

%Let $k$ be an infinite field and $X$ be a $k$-smooth irreducible affine variety.
%Replacing $k$ by its algebraic closure in $k[X]$ we may assume that $X$ is smooth affine and geometrically
%irreducible over $k$.
%Let
%$x_1,x_2, \dots , x_n$
%be a finite family of points
%on $X$. Let $G$ be a simple simply connected group scheme over $G$
%which contains a split rank one torus $\Bbb G_{m,X}$.
%For each $i=1,2, \dots, n$ let $G(x_i)$ be the fibre of $G$ over the point $x_i$,
%that is
%$G(x_i)=G \times_X x_i$.
%Note that for each
%$i=1,2, \dots, n$
%the $k(x_i)$-group scheme $G(x_i)$ is a
%simple simply connected $k(x_i)$-group scheme
%which contains a split rank one torus
%$\Bbb G_{m,x_i}$.

\begin{cor}\label{BundlesOnP1}
Let $l$ be a field, and let $G_l$ be an isotropic simply connected semi-simple $l$-group
scheme with a parabolic $l$-subgroup $Q_l$.
%% In particular, $G_l$ contains $\Bbb G_{m,l}$.
Let $P$ be a $G_l$-bundle over $\Aff^1_l$. Further, let $f(t)\in
l[t]$ be a non-constant polynomial, $\varphi: G_{\Aff^1_f}\to
P_{\Aff^1_f}$ be a principal $G_l$-bundle isomorphism and let
$P(\varphi,f)$ be the corresponding principal $G_l$-bundle on the
projective line $\Pro^1_l$. Then there exist $h(t)\in l[t]$ and
$u\in E_{Q_l}(l[t]_{tfh})$ such that the principal $G_l$-bundle
$P(\varphi \circ u,tfh)$ is trivial over $\Pro^1_l$.
\end{cor}

\begin{proof}
By Lemma \ref{RR} there exists an $\alpha \in G_l(l[t]_f)$ such
that the principal $G_l$-bundle $P(\phi \circ \alpha , f)$ is
trivial.
\par
Let $f(t)=t^rg(t)$ be the unique factorisation such that $g(t)\in
l[t]$ and $g(0)\neq 0$. By Lemma~\ref{AlphaUtimesbeta} there exist an element
$h(t)\in l[t]$ with $h(0)\neq 0$ and elements
$$ u \in G_l(l[t]_{tfh})^{+},\qquad
\beta\in G_l(l) $$
\noindent
such that
\begin{equation}
\label{AlphaEqUtimesBeta} \alpha=u\beta\in
G_l(l[t^{-1},t]_{g^*h^*}).
\end{equation}
The following chain of principal $G_l$-bundle isomorphisms
completes the proof:
$$ G_l \times_{\text{Spec}(l)} \Pro^1_l =P(\varphi \circ \alpha, f)=
P(\varphi \circ \alpha,\ tfh) = P(\varphi \circ u \circ \beta,
tfh) \cong P(\varphi \circ u, \ tfh). $$
\noindent
Here all the equalities are obvious. The last isomorphism holds by Remark~\ref{PropertiesPofPhiAndf},
since $\beta\in G_l(l[t^{-1}]_{g^*h^*})$.

\end{proof}

Let $B^{\prime}$ be a Noetherian semi-local ring.
%$X$ be a $k$-smooth
%irreducible affine variety. Replacing $k$ by its algebraic closure
%in $k[X]$ we may assume that $X$ is smooth affine and
%geometrically irreducible over $k$. Let $x_1,x_2,\dots,x_n$ be a
%finite family of points on $X$. Let $\mathcal O=\mathcal
%O_{X,\{x_1,x_2,\dots,x_n\}}$ be the semi-local ring of the family
%$x_1,x_2,\dots,x_n\in X$ and let $U$ for $\text{Spec}(\mathcal
%O)$.
Let $G$ be a simple simply connected $B^{\prime}$-group scheme.
%By
%$G(x_i)$, $1\le i\le n$, we denote the fibre of $G$ over the point
%$x_i$, in other words, $G(x_i)=G\times_X x_i$. Note that for each
%$i=1,2,\dots,n$ the $k(x_i)$-group scheme $G(x_i)$ is an isotropic
%simple simply connected $k(x_i)$-group scheme.
%% which contains a split rank one torus $\Bbb G_{m,x_i}$.
Let $\frak{m}_i\subseteq  B^{\prime}$, $i=1,2, \dots, n$, be all maximal ideals of $B^{\prime}$. Let $J$ be the intersection of all
$\frak{m}_i$, $1\le i\le n$. Then
$$
l:=B^{\prime}/J=l_1\times l_2\times\dots\times l_n,$$
where $l_i=B^{\prime}/\frak{m}_i$.
%% Set $l=l_1\timesl_2\times\dots\times l_n$, then $k[X]/J = l$.
Let $G_l=G\otimes_{B^{\prime}} l$ be the fibre of $G$ over $\Spec(l)$.
%Let
%$\mathcal O:= \mathcal O_{X, \{x_1,x_2,\dots,x_n\}}$
%be semi-local ring of the family
In the sequel we write $\Pro^1$ and $\Aff^1$ for
$\Pro^1_{B^{\prime}}$ and $\Aff^1_{B^{\prime}}$ respectively, whereas $\Pro^1_l$
and $\Aff^1_l$ denote the projective line and the affine line over
$l$.

Let $f\in B^{\prime}[t]$ be a monic polynomial, and let $P$ be a
principal $G_{B^{\prime}}$-bundle over $\Aff^1$ such that
$P_{\Aff^1_f}$ is trivial. Let $\varphi: G_{\Aff^1_f} \to
P_{\Aff^1_f}$ be a principal $G$-bundle isomorphism, and let
$P(\varphi, f)$ be the corresponding principal $G$-bundle on
$\Pro^1$ (see Definition \ref{PofPhiAndf}).

\begin{thm}\label{NewBundle}
Assume that the group scheme $G$ over $B^{\prime}$ is
isotropic, simple and simply connected.
%% and $G$ contains an $\mathcal O$-split torus $\Bbb G_{m,\mathcal O}$.
Then there exist a monic polynomial $h(t)\in B^{\prime}[t]$ and an
element $\alpha\in G(B^{\prime}[t]_{tfh})$ such that the principal
$G$-bundle $P(\varphi\circ\alpha,tfh)$ satisfies the condition
\begin{itemize}
\item[{\rm(}i{\rm)}] $P(\varphi\circ\alpha,tfh)|_{\Pro^1_l}$ is a
trivial principal $G_l$-bundle over the projective line
$\Pro^1_l$.
\end{itemize}
\end{thm}

\begin{proof}
We denote by $\overline f$ the image of $f$ in $l[t]$, by $\overline P$ the restriction of $P$ to
$\Aff^1_l$, by $\overline P(\overline\varphi,\overline f)$ the
restriction of $P(\varphi,f)$ to the projective line $\Pro^1_l$, etc. Let $Q$ be a parabolic $B^{\prime}$-subgroup
of $G$. By
%Lemma
%\ref{AlphaUtimesbeta}
%and
Corollary~\ref{BundlesOnP1}
there exist a monic polynomial
$\overline h(t)\in l[t]$ such that $\overline h(0)\in l^{\times}$, and
an element
$$ u\in E_{Q_l}\big(l[t]_{t\overline f\,\overline h}\big)\le
%% \ \langle U^+_l(l[t]_{t\overline f \overline h}),
%% U^-_l(l[t]_{t\overline f\,\overline h})\rangle \ \subset
G_l\big(l[t]_{t\overline f\,\overline h}\big) $$
\noindent
such that the principal $G_l$-bundle $\overline
P(\overline\varphi\circ u,t\overline f\,\overline h)$ is trivial
over $\Pro^1_l$.
\par
Choose a monic polynomial $h(t)\in B^{\prime}[t]$ of degree equal
to the degree of $\overline h(t)$ and such that $h(t)$ modulo $J$
coincides with $\overline h(t)$. Clearly, the homomorphism of $B^{\prime}$-algebras
$B^{\prime}[t]_{tfh}\to l[t]_{t\overline f\,\overline h}$ is surjective.
By Lemma \ref{Surjectivity} it
induces a surjective group homomorphism
$$ E_{Q}\big(B^{\prime}[t]_{tfh}\big)\to E_{Q}\big(l[t]_{t\overline f\,\overline h}\big)=
E_{Q_l}\big(l[t]_{t\overline f\,\overline h}\big). $$
%% \langle U^+(\mathcal O[t]_{tfh}), U^+(\mathcal O[t]_{tfh})\rangle
%% \to \ \langle U^+_l(l[t]_{t\overline f \overline h}),\
%% U^-(l[t]_{tfh})\rangle.
%\noindent
Thus, there exists an $\alpha\in E_{Q}(B^{\prime}[t]_{tfh})\le G(B^{\prime}[t]_{tfh})$ such that
$\alpha$ equals $u$ modulo $J$;  we write $\overline\alpha=u$.
%$ u\in E_{P_l}(l[t]_{t\overline f\,\overline h})\le
%%\langle U^+_l(l[t]_{t\overline f\overline h}),
%% U^-_l(l[t]_{t\overline f\overline h})\rangle \subset
%G_l(l[t]_{t\overline f\,\overline h}) $$
%\noindent
%modulo $J$.
\par
Consider the $G$-bundle $P(\varphi \circ \alpha, tfh)$. We claim that
its restriction to the projective line $\Pro^1_l$ is trivial.
Indeed, one has the following chain of equalities
of principal $G_l$-bundles over $\Pro^1_l$:
$$ \overline P(\overline{\varphi\circ\alpha},\overline{tfh})=
\overline P(\overline\varphi\circ\overline\alpha,t\overline
f\,\overline h)=\overline P(\overline\varphi\circ u,t\overline
f\,\overline h), $$
\noindent
where the principal $G_l$-bundle $\overline
P(\overline\varphi\circ u,t\overline f\overline h)$ is trivial
over $\Pro^1_l$.
\end{proof}
We keep the same notation as introduced before Theorem~\ref{NewBundle}.
\begin{thm}
\label{ConstantOnFibreIsConstant}
Assume that the Noetherian semi-local ring $B^{\prime}$ contains a field $k$.
%Let $k$ be a field, and let $B^{\prime}$ be a semi-local Noetherian algebra over $k$.
Let $G$ be a not necessarily isotropic simple simply connected $B^{\prime}$-group scheme. Let $E$ be a principal
$G$-bundle over $\Pro^1$ whose restriction to the closed fibre
$E_{\Pro^1_l}$ is trivial. Then $E$ is of the form:
$E=\pr^*(E_0)$, where $E_0$ is a principal $G$-bundle over
$\Spec(B^{\prime})$ and $\pr:\Pro^1\to\Spec(B^{\prime})$ is the
canonical projection.
\end{thm}
\begin{proof}
%The proof of this Theorem is rather standard,
%and for the most part follows \cite{R2}. However, our group scheme
%$G$ does not come from the ground field $k$. Therefore, we have to
%somewhat modify Raghunathan's arguments.
See Appendix,~\ref{horrockstypetheorem}.
%% rather than literally repeat them .
\end{proof}

Let us state an important corollary of the above theorems.
\begin{cor}
%[{\bf=Theorem \ref{TrivialityOfP_t}}]
\label{TrivialOnFibraIsTrivialLocal} Let $k$ be a field, and let $B^{\prime}$ be a semi-local Noetherian algebra over $k$.
Let $G$ be an isotropic simple simply connected $B^{\prime}$-group scheme. Further, let $P$ be a principal $G$-bundle over
$\Aff^1$. Assume that there exists a monic polynomial
$f\in B^{\prime}[t]$ such that the principal $G$-bundle
$P_{\Aff^1_f}$ is trivial.
%Let $\varphi: G_{\Aff^1_f} \to P_{\Aff^1_f}$
%be a principal $G$-bundle isomorphism and let
%$P(\varphi, f)$ be the corresponding principal $G$-bundle on
%$\Pro^1_{\mathcal O}$ (see Definition
%\ref{PofPhiAndf}).
Then the principal $G$-bundle $P$ is trivial.
%In other words,
%there exists a $G$-bundle isomorphism
%$$ G\times_U\Aff^1\cong P. $$
\end{cor}
\begin{proof}
%[Proof of Corollary \ref{TrivialOnFibraIsTrivialLocal}]
%%{ConstantOnFibreIsConstant}]
Let $f\in B^{\prime}[t]$ be a monic polynomial such that the
principal $G$-bundle $P_{\Aff^1_f}$ is trivial. Choose a principal
$G$-bundle isomorphism $\varphi:G_{\Aff^1_f}\to P_{\Aff^1_f}$. By
Theorem \ref{NewBundle} there exists a monic polynomial
$h(t)\in B^{\prime}[t]$ and an element $\alpha\in G(B^{\prime}[t]_{tfh})$ such that the restriction
$P(\varphi\circ\alpha,tfh)|_{\Pro^1_l}$ of the principal
$G$-bundle $P(\varphi\circ\alpha,tfh)$ to the projective line
$\Pro^1_l$ is a trivial principal $G_l$-bundle.
\par
By Theorem \ref{ConstantOnFibreIsConstant} the principal
$G$-bundle $P(\varphi\circ\alpha,tfh)$ is of the form:
$P(\varphi\circ\alpha,tfh)=\pr^*(P_0)$, where $P_0$ is a principal
$G$-bundle over $\text{Spec}(B^{\prime})$. Note that
$$ G|_{\{\infty\}\times {\text{Spec}(B^{\prime})}}\cong P(\varphi\circ\alpha,
tfh)|_{\{\infty\}\times {\text{Spec}(B^{\prime})}}, $$
\noindent
%where $U=\text{Spec}(B^{\prime})$
that is the restriction of
$P(\varphi\circ\alpha, tfh)$ to $\{\infty\}\times \text{Spec}(B^{\prime})$ is trivial.
Thus
$$ G_{\Pro^1}\cong P(\varphi\circ\alpha,tfh). $$
\noindent
Since the original principal $G$-bundle $P$ over $\Aff^1$ is
isomorphic to $P(\varphi\circ\alpha,tfh)|_{\Aff^1}$, it follows
that $P$ is trivial. This finishes the proof.

\end{proof}

\begin{cor}[{\bf=Theorem \ref{MainThmArithmetic}}]
\label{TrivialOnFibraIsTrivialGlobal}
Let $k$ be a field and let $B$ be a Noetherian $k$-algebra.
Assume that a group scheme $G$ over $B$ is
simple, simply connected and isotropic.
%Let $G$ be an $B$-group
%scheme satisfying the same assumptions as in Theorem
%\ref{NewBundle}.
Further, let $P$ be a principal $G$-bundle over
$\Aff^1_B$. Assume that there exists a monic polynomial
$f\in B[t]$ such that the principal $G$-bundle
$P_{(\Aff^1_B)_f}$ is trivial and $f(1) \in B$ is invertible.
%Let $\varphi: G_{\Aff^1_f} \to P_{\Aff^1_f}$
%be a principal $G$-bundle isomorphism and let
%$P(\varphi, f)$ be the corresponding principal $G$-bundle on
%$\Pro^1_{\mathcal O}$ (see Definition
%\ref{PofPhiAndf}).
Then the principal $G$-bundle $P$ is trivial.
%In other words,
%there exists a $G$-bundle isomorphism
%$$ G\times_U\Aff^1\cong P. $$
\end{cor}
\begin{proof}
%[Proof of Corollary \ref{TrivialOnFibraIsTrivialGlobal}]
It is routine to prove that there is a closed $B$-group scheme embedding
$G \hra GL_{N,B}$ for an $N > 0$. Since $f(1)$ is invertible, the principle $G$-bundle $P$
is trivial at the closed subscheme $\{1\} \times \Spec (B) \subset \Aff^1_B$.
By Corollary~\ref{TrivialOnFibraIsTrivialLocal}, for any maximal
ideal $m$ of $B$, the bundle $P_{\Aff^1_{B_m}}$ is trivial too.
Now~\cite[Korollar 3.5.2]{Mo}
%and
%Corollary
%\ref{TrivialOnFibraIsTrivialLocal}
completes the proof of
Corollary
\ref{TrivialOnFibraIsTrivialGlobal}.

\end{proof}
%\begin{proof}[Proof of Theorem \ref{TrivialityOfP_t}]
%It follows immediately from Corollary
%\ref{TrivialOnFibraIsTrivial}
%% \ref{ConstantOnFibreIsConstant}
%with the use of the celebrated Popescu theorem \cite{P}, see also
%\cite[Section 7, Popescu's Theorem]{OP}.
%\end{proof}

%We already mentioned the
%existence of a finite surjective morphism of $U$-schemes
%$\Pi^* \circ \theta: \mathcal X^{\prime} \to \Aff^1 \times U$.
%It remains to check that $q^{\prime}_U$ is a smooth morphism. This is the case
%since so are $q_U$ and $\theta$.
%REMARK: THE FACT THAT $q^{\prime}_U$ has geometrically irreducible fibres IS NOT CLEAR.
%IN THE NEW REDUCTION OF THE TEXT WE NEED NOT the even irreducibility of closed fibres !!!

\section{Proofs of Theorems \ref{MainThmGeometric} and \ref{MainThm}, and of Corollary~\ref{HominvarCor}}
\label{ProofOfMainThm}

%the case, where the ring $R$
%is the semi-local ring $\mathcal O_{X, \{x_1,x_2,\dots,x_n\}}$ of
%a smooth $k$-variety $X$ over an infinite field $k$ at a finite
%family of points $x_1,x_2,\dots,x_n\in X$.
%\par\smallskip
%Theorem
%\ref{MainThmGeometric}.

\begin{proof}[Proof of Theorem~\ref{MainThmGeometric}]
Substitute to Theorem \ref{MainThmArithmetic}
$B=\mathcal O \otimes_k A$, $P_t:=\mathcal G_t$, $h(t)=f(t) \otimes 1$
from Theorem \ref{MainHomotopy}. By Theorem
\ref{MainThmArithmetic}
the $G$-bundle $\mathcal G_t$ is trivial.
Now by the item (ii) of Theorem \ref{MainHomotopy} the original $G$-bundle $\mathcal G$ is trivial.

\end{proof}
\begin{proof}[Proof of Theorem~\ref{MainThm}]
%Theorem~\ref{MainThm} follows from Theorem~\ref{MainThmGeometric} by means of essentially the same argument as
%in~\cite[Proof of Theorem 1, pages 4--5]{FP}.

Let $\mathcal G$ be a principal $G$-bundle over $R\otimes_\ZZ A$ that becomes trivial over
$K\otimes_\ZZ A$.
Clearly, there is a non-zero $f\in R$
such that $\mathcal G$ is trivial over $R_f\otimes_{\ZZ} A$.

Let $k'$
be the prime subfield of $R$.
It follows from Popescu's
theorem~\cite{P,Sw} that $R$ is a filtered inductive limit of smooth $k'$-algebras $R_\alpha$.
Then there exist an index $\alpha$, a reductive group scheme $G_\alpha$ over $R_\alpha$, a principal $G_\alpha$-bundle
$\mathcal G_\alpha$ over $R_\alpha\otimes_{\ZZ} A$, and an element $f_\alpha\in R_\alpha$ such that
$G = G_\alpha\times_{\Spec(R_\alpha)} \Spec(R)$,
$\mathcal G\cong {\mathcal G}_\alpha\times_{\Spec(R_\alpha\otimes_{\ZZ} A)} \Spec(R\otimes_\ZZ A)$ as principal $G$-bundles,
$f$ is the image of $f_\alpha$ under the map
$\phi_\alpha : R_\alpha\to R$, and ${\mathcal G}_\alpha$ is trivial over $(R_\alpha)_{f_\alpha}\otimes_\ZZ A$.

%Let $m_1,\ldots,m_n$ be all maximal ideals of $R$.
If the field $k'$ is infinite, then for each maximal ideal $m_i$ in $R$ ($i = 1,\ldots, n$)
set $p_i=\phi_\alpha^{-1}(m_i)$. The map $\phi_\alpha$ induces a map
of semi-local rings $(R_\alpha)_{p_1,\ldots,p_n}\to R$.
%We have natural isomorphisms of $R$-algebras
%$$
%R_\alpha\otimes_\ZZ A\cong R_\alpha\otimes_{k'}(k'\otimes_\ZZ A)\quad\mbox{and}\quad K\otimes_\ZZ A\cong K\otimes_{k'}(k'\otimes_\ZZ A).
%$$
%Therefore, after changing $A$ to $k'\otimes_\ZZ A$, we can assume that we are given
%a principal $G$-bundle $\mathcal G$ over $R\otimes_{k'} A$ that becomes trivial over
%$K\otimes_{k'} A$, and we need to show that this bundle is trivial already over $R\otimes_{k'} A$.
Since the principal $G_\alpha$-bundle $\mathcal G_\alpha$ is trivial over
$(R_\alpha)_{f_\alpha}\otimes_\ZZ A\cong(R_\alpha)_{f_\alpha}\otimes_{k'} (k'\otimes_\ZZ A)$,
by Theorem~\ref{MainThmGeometric} the bundle
${\mathcal G}_\alpha$ is trivial over
$(R_\alpha)_{p_1,\ldots,p_n}\otimes_{k'} (k'\otimes_\ZZ A)\cong (R_\alpha)_{p_1,\ldots,p_n}\otimes_\ZZ A$.
Whence the $G$-bundle $\mathcal G$ is trivial over $R\otimes_\ZZ A$.

Now consider the case where the field $k'$ is finite. Since $R$ contains an infinite field by the assumption of the theorem,
$R$ also contains a field $k'(t)$ of rational functions in one variable $t$ over $k'$. Set
$R'_\alpha=R_\alpha\otimes_{k'}k'(t)$, then the map $\phi_\alpha$ can be decomposed as follows
$$
R_\alpha\to R_\alpha\otimes_{k'}k'(t) = R'_\alpha\xrightarrow{\psi_\alpha} R.
$$
Set $G'_\alpha=G_\alpha\times_{\Spec(R_\alpha)}\Spec(R'_\alpha)$,
${\mathcal G'}_\alpha={\mathcal G}_\alpha\times_{\Spec(R_\alpha\otimes_\ZZ A)}\Spec(R'_\alpha\otimes_\ZZ A)$,
$f'_\alpha=f_\alpha\otimes 1\in R'_\alpha$. Then $R'_\alpha$ is a smooth $k'(t)$-algebra, and the principal
$G'_\alpha$-bundle ${\mathcal G'}_\alpha$
is trivial over $(R'_\alpha)_{f'_\alpha}\otimes_\ZZ A$.
%For each maximal ideal $m_i$ in $R$ ($i = 1,\ldots, n$)
%set $p_i=\psi_\alpha^{-1}(m_i)$. The map $\psi_\alpha$ induces a map
%of semi-local rings $(R'_\alpha)_{p_1,\ldots,p_n}\to R$. By Theorem~\ref{MainThmGeometric} the principal $G'_\alpha$-bundle
%${\mathcal G'}_\alpha$ is trivial over $(R'_\alpha)_{p_1,\ldots,p_n}\otimes_{k'(t)} A$.
Arguing exactly as in the previous case with the field $k'(t)$ instead of $k'$, we conclude, by means of Theorem~\ref{MainThmGeometric},
that the $G$-bundle $\mathcal G$ is trivial over $R\otimes_\ZZ A$.
\end{proof}
%is subdivided into two
%parts:
%\par\smallskip
%(2) a {\it geometric part\/}, and
%\par\smallskip
%(3) a {\it group part}.
%\par\smallskip

%To prove this result we prove the following Lemma and its version for semi-local rings. The Lemma inspired by
%\cite[Cor.1.8]{R1}
%
%\begin{lem}
%Let $R$ and $G$ be as above. Let $\textit{m}$ be the maximal ideal of $R$ and
%$l=R/\textit{m}$
%be the residue field.
%Let $S \subseteq R[t]$ be the multiplicative system of all unitary
%polynomial in $R[t]$ and let
%$$\pi: A=R[t]_S \to l[t]_S=l(t)=L$$
%be the canonical projection, induced by the projection
%$R \to l$. Let $\mathcal O=l[t]_{(1/t)}$
%be the localization of the ring $l[1/t]$ at the maximal ideal $1/t \cdot l[1/t]$.
%Let $U^+$ and $U^-$ be $R$-group schemes, which are the unipotent radicals of two opposite
%parabolic subgroups of $G$ over $R$.
%
%Let $\pi: G(A) \to G(L)$ be a group homomorphism induced by the projection
%$\pi: A \to L$. Then $G(L)=\pi (<U^+(A), U^-(A)>) \cdot G(\mathcal O)$.
%\end{lem}

%The geometric part of the proof

\begin{proof}[Proof of Corollary~\ref{HominvarCor}]
Consider the commutative diagram
\begin{equation}
\xymatrix{
H^1_{\et}(R[t],G)\ar[rr]^{t=0}\ar[d]_{} &&
H^1_{\et}(R,G)\ar[d]^{}  &\\
H^1_{\et}(K[t],G)\ar[rr]^{t=0}&& H^1_{\et}(K,G).
}
\end{equation}
Since $K$ is perfect, the bottom arrow is bijective by the main result of~\cite{RR}. Therefore,
any element $\xi\in H^1_{\et}(R[t],G)$ having trivial image in $H^1_{\et}(R,G)$ also has
trivial image in $H^1_{\et}(K[t],G)$. By Theorem~\ref{MainThm}, for any maximal ideal $m\subseteq R$ the map
$$
H^1_{\et}(R_m[t],G)\to H^1_{\et}(K[t],G)
$$
has trivial kernel. Therefore, for any maximal ideal $m$, the image of $\xi$ in $H^1_{\et}(R_m[t],G)$ is trivial as well.
By~\cite[Korollar 3.5.2]{Mo} this implies that $\xi$ is trivial.
% The left vertical
%line has trivial kernel by~Corollary~\ref{cor:nonlocal-line}.

\end{proof}

\section{Semi-simple case}
\label{SectSemi-Simple}

In the present Section we show how Theorem~\ref{MainThmGeometric}
extends to the case of semi-simple simply connected groups; this is Theorem~\ref{MainThmGeometricSemi-Simple}
below. One readily sees that
Theorem~\ref{MainThm} and Corollary~\ref{HominvarCor} extend to semi-simple simply connected groups
as well, once we substitute the isotropy condition imposed in these statements by the same one as
in Theorem~\ref{MainThmGeometricSemi-Simple}.

By~\cite[Exp. XXIV 5.3, Prop. 5.10]{SGA3} the category of semi-simple simply connected
group schemes over a Noetherian domain $R$ is semi-simple. In
other words, each object has a unique decomposition into a product
of indecomposable objects. Indecomposable objects can be described
as follows. Take a domain $R^{\prime}$ such that $R\subseteq
R^{\prime}$ is a finite \'{e}tale extension and a simple
simply connected group scheme $G^{\prime}$ over $R^{\prime}$. Now,
applying the Weil restriction functor $\R_{R^{\prime}/R}$ to the
$R$-group scheme $G^{\prime}$ we get a simply connected $R$-group
scheme $\R_{R^{\prime}/R}(G^{\prime})$, which is
%we get $R$-group schemes which are
%% The latter $R$-group scheme is
an indecomposable object in the above category. Conversely, each
indecomposable object can be constructed in this way.

\begin{defn}\label{DefSSIsotropy}
We say that an indecomposable semi-simple simply connected
group schemes
$H=\R_{R^{\prime}/R}(H^{\prime})$
over a Noetherian domain $R$ is \emph{isotropic} if $H^{\prime}$ is isotropic.
\end{defn}

%Let $k$ be an infinite field. Let
%$\mathcal O$ be the semi-local ring of finitely many {\bf closed} points on a
%smooth irreducible $k$-variety $X$ and let $K$ be its field of
%fractions.
%Let $G$ be an isotropic simple simply connected group
%scheme over $\mathcal O$.
%Then for any Noetherian $k$-algebra $A$
%% containing a split rank $1$ torus $\Bbb G_m$.
%the map
%$$ H^1_{\text{\rm\'et}}(\mathcal O \otimes_k A,G)\to H^1_{\text{\rm\'et}}(K \otimes_k A,G), $$
%\noindent
%induced by the inclusion $\mathcal O$ into $K$, has trivial kernel.

\begin{thm}
\label{MainThmGeometricSemi-Simple}
Let $k$ be an infinite field. Let
$\mathcal O$ be the semi-local ring of finitely many closed points on a
smooth irreducible $k$-variety $X$ and let $K$ be its field of
fractions.
 Let $G$ be a semi-simple simply connected
$\mathcal O$-group scheme all of whose indecomposable factors are
isotropic in the sense of Definition~\ref{DefSSIsotropy}. Then for any Noetherian $k$-algebra $A$
the map
$$ H^1_{\text{\rm\'et}}(R \otimes_k A,G)\to H^1_{\text{\rm\'et}}(K \otimes_k A,G), $$
\noindent
induced by the inclusion $R$ into $K$, has trivial kernel.
\end{thm}

\begin{proof}
Take a decomposition of $G$ into indecomposable factors
$G=G_1\times G_2\times\dots\times G_r$. Clearly, it suffices to
check that for each index $i$ the kernel of the map
$$ H^1_{\text{\'et}}(R \otimes_k A,G_i)\to H^1_{\text{\'et}}(K \otimes_k A,G_i) $$
\noindent
is trivial. We know that there exists a finite \'{e}tale extension
$R^{\prime}_i/R$ such that $R^{\prime}_i$ is a domain and the Weil
restriction $\R_{R^{\prime}_i/R}(G^{\prime}_i)$ coincides with
$G_i$.
\par
The Faddeev---Shapiro Lemma \cite[Exp. XXIV Prop. 8.4]{SGA3} states
that there is a canonical isomorphism
$$
H^1_{\text{\'et}}\big(R \otimes_k A, \R_{R^{\prime}_i/R}(G^{\prime}_i)\big)
\cong H^1_{\text{\'et}}\big(R^{\prime} \otimes_k A,G_i\big)
$$
\noindent
that preserves the distinguished point.
To complete the proof, it only remains to apply Theorem
\ref{MainThm} to the semi-local regular ring $R^{\prime}_i$, its
fraction field $K_i$ and the simple $R^{\prime}_i$-group scheme
$G^{\prime}_i$.

\end{proof}

%\begin{thm}
%\label{MainThmSemi-SimpleGeometric} Let $k$ be an infinite field.
%Let $\mathcal O$ be the semi-local ring of finitely many points on
%a smooth irreducible $k$-variety $X$ and let $K$ be its field of
%fractions. Let $G$ be a semi-simple simply connected group scheme
%$G$ all of whose indecomposable factors are isotropic.
%% contains a split rank $1$ torus $\Bbb G_m$.
%Then the map
%$$ H^1_{\text{\rm\'et}}(R,G)\to H^1_{\text{\rm\'et}}(K,G) $$
%\noindent
%induced by the inclusion of $R$ into $K$ has trivial kernel.
%\end{thm}

%\begin{proof}
%Use Theorem
%\ref{MainThmGeometric}
%and argue literally as in the proof of Theorem
%\ref{MainThmSemi-Simple}

%\end{proof}

%\subparagraph{Acknowledgments} The authors are very grateful to
%Konstantin Pimenov and Victor Petrov for useful discussions on the subject of the
%present article.
%The  author thanks very much for the support the
%RTN-Network HPRN-CT-2002-00287,
%the RFFI-grant 03-01-00633a,
%the Swiss National Science Foundation
%and INTAS-03-51-3251.

%%%%%%%%%%%%%%%%%%%%%%%%%%%%%%%%%%%%%%%%%%%%%%%%%%%%%%%%%%%%%

\section{Appendix}
\subsection{Almost elementary fibration}
\label{ArtinsNeighb}

In this section we prove
%Theorem~\ref{GeneralSection},
%as well as
Propositions~\ref{ArtinsNeighbor}
and~\ref{CartesianDiagram}.

\begin{proof}[Proof of Proposition~\ref{ArtinsNeighbor}]
The proof almost literally follows the proof of the original Artin's result~\cite[Exp. XI, Prop. 3.3]{LNM305}.
Shrinking $X$, may assume that $X \subset \Aff^r_k$ is affine and still contains the points
$x_1,x_2,\dots,x_n$. Set $x:= \coprod\limits^n_{j=1}x_i$.
Let $X_0$ be the closure of $X$ in $\Pro^r_k$.
Let $\bar X$ is the normalization of $X_0$ and set $Y=\bar X - X$ with the induced reduced structure.
Let $Z \subset X$ be a subset of $\bar X$ consisting of all non-regular points of $\bar X$.
By \cite[Cor. 6.12.5]{EGAIV2} the set $Z$ is Zariski closed in $\bar X$. Since $\bar X$ is normal
we conclude that
$\dim Z \leq n-2$.
Since $X$ is $k$-smooth
one has an inclusion
$Z \subset Y$.
Summarizing one has
\begin{itemize}
\item[(i)]
$Z \subset Y$,
\item[(ii)]
$\dim \bar X = \dim X = n$,
\item[(iii)]
$\dim Y = n-1$,
\item[(iv)]
$\dim Z \leq n-2$.
\end{itemize}
Shrinking $X$ and following Artin's procedure one can construct a diagram of the form
\begin{equation}
\label{SquareDiagram}
    \xymatrix{
     X\ar[drr]_{p}\ar[rr]^{j}&&
\overline X^{\prime} \ar[d]_{\overline p}&&Y^{\prime} \ar[ll]_{i}\ar[lld]_{q} &\\
     && S  &\\    }
\end{equation}
%such that $x \subset X$,
%$S$ is an affine open subset in $\Pro^{n-1}_k$
%and for the algebraic closure $\bar k$ of $k$ the diagram
%\begin{equation}
%\label{SquareDiagram}
%    \xymatrix{
%     X\otimes_k \bar k \ar[drr]_{p}\ar[rr]^{j}&&
%(\overline X^{\prime})\otimes_k \bar k  \ar[d]_{\overline p}&&(Y^{\prime}\otimes_k \bar k)_{red}  \ar[ll]_{i}\ar[lld]_{q} &\\
%     && S\otimes_k \bar k   &\\    }
%\end{equation}
subject to conditions (i), (ii) of Definition~\ref{DefnElemFib}
and such that $x \subset X$ and $S$ is an affine open subset in $\Pro^{n-1}_k$
and $i$ is a closed imbedding and $Y^{\prime}$ is a regular scheme.
Moreover, the restriction of
$q \otimes_k \bar k: Y^{\prime} \otimes_k \bar k \to S\otimes_k \bar k$
to the reduced subscheme
$(Y^{\prime} \otimes_k \bar k)_{red}$ is a finite \'etale morphism
all of whose fibres are non-empty
(here $\bar k$ is the algebraic closure of $k$).
In this case for each irreducible component
$Y^{\prime}_r$ of $Y^{\prime}$
the restriction
$q|_{Y^{\prime}_r}: Y^{\prime}_r \to S$ is a finite surjective morphism.
Since
$Y^{\prime}_r$, $S$ are regular irreducible schemes of the same dimension,
the morphism
$q|_{Y^{\prime}_r}$
is finite flat
(see Grothendieck~\cite[Cor. 17.18]{E}).
Thus $q$ is subject to the condition (iii) of Definition~\ref{DefnElemFib}.
Finally, the ideal sheaf $I_{Y^{\prime}}$ defining the closed subscheme
$Y^{\prime}$ in $\overline X^{\prime}$ is locally principal. In fact, $S$ is regular
and $\overline p$ is smooth. So, $\overline X^{\prime}$ is regular.
The closed subscheme $Y^{\prime}$ is regular of pure codimension one in
$\overline X^{\prime}$. Thus $I_{Y^{\prime}}$ is locally principal.
Whence the Proposition.
\end{proof}

\begin{proof}[Proof of Proposition~\ref{CartesianDiagram}]
To prove this Proposition it suffices to construct a finite surjective $S$-morphism
$$\bar \pi: \overline X \to \Pro^1 \times S$$
such that
$Y_{red}=\bar \pi^{-1}(\{\infty\} \times S )$
set-theoretically. To do that, we first note that, under the hypotheses of the Proposition,
the closed subscheme $Y$ of $\overline X$
is a locally principal divisor.
We will construct a desired $\bar \pi$ using two sections $t_0$ and $t_1$ of the sheaf
$\mathcal O(nY)$ for a sufficiently large $n$. Assume that $t_0$ and $t_1$ are
such that the vanishing locus of $t_0$ is $nY$ and the vanishing locus of $t_1$ does not intersect $Y$.
Then the pair $t_0,t_1$ defines a regular map
$\varphi:=[t_0:t_1]: \overline X \to \Pro^1$.
Set
$\bar \pi = (\varphi,\bar p): \overline X \to \Pro^1 \times S$.
Clearly,
$\bar \pi$ is an $S$-morphism of the $S$-schemes. It is a projective morphism
since both $S$-schemes are projective $S$-schemes. It is a quasi-finite surjective morphism.
In fact, for each point $s \in S$ the morphism $\bar \pi$ induces a non-constant morphism
$\overline X_s \to \Pro^1_{s}$
of two $k(s)$-smooth geometrically irreducible projective $k(s)$-curves.
Thus $\bar \pi$ is finite surjective as a quasi-finite projective morphism.
It remains to find an appropriate integer $n$ and two sections $t_0$ and $t_1$ with the above properties.

Firstly, for each point $s$ of the scheme $S$ set
$\overline X_s:=(\bar \pi)^{-1}(s)$
scheme-theoretically, and note that
$\overline X_s$ is a
$k(s)$-smooth geometrically irreducible projective $k(s)$-curve. The morphism $\bar \pi$ is smooth.
In particular, it is flat. Whence the function
$s \mapsto \chi(\overline X_s, \mathcal O_{\overline X_s})$
is constant by~\cite[Ch. II, Sect. 5, Cor. 1]{Mu}. The latter means that the genus
$g(\overline X_s)$
is the same for all points $s \in S$. Set
$g = g(\overline X_s)$.
By the assumption $Y$ is finite flat over $S$ and $S$ is semi-local. Let $r$ be the rank of the free
$\Gamma(S, \mathcal O_S)$-module $\Gamma(Y, \mathcal O_Y)$.

Assume that $n \geq 2g-1$, then
$h^0(\overline X_s, \mathcal O_{\overline X_s}(nY_s))=\chi(\overline X_s, \mathcal O_{\overline X_s}(nY_s))=rn-g+1$.
Let
$\mathcal E_n:= \bar p_*(\mathcal O_{\overline X_s}(nY_s)))$.
By~\cite[Ch. II, Sect. 5, Cor. 1]{Mu} and~\cite[Ch. II, Sect. 5, Lem. 1]{Mu}
the sheaf $\mathcal E_n$ on $S$
is locally free of rank
$rn-g+1$,
and for each point $s \in S$ the canonical map
$\mathcal E_n \otimes_{\mathcal O_S} k(s) \xrightarrow{can} H^0(\overline X_s, \mathcal O_{\overline X_s}(nY_s))$
is an isomorphism.

Let $s = \coprod s_i$, where $s_i$ are all closed points of the semi-local scheme $S$. Let
$k(s) = \prod k(s_i)$, where $k(s_i)$ denotes the residue field of the point $s_i$.
Consider the commutative diagram
\begin{equation}
\label{Diagram3}
    \xymatrix{
H^0(S, \mathcal E_n) \ar[rr]^{\id}\ar[d]_{\alpha} && H^0(\overline X, \mathcal O_{\overline X}(nY)) \ar[d]^{\beta}& \\
\mathcal E_n \otimes_{\mathcal O_S} k(s) \ar[rr]^{\can} && H^0(\overline X_s, \mathcal O_{\overline X_s}(nY_s)),      & \\  }
\end{equation}
where $\alpha$, $\beta$, and $\can$ are the canonical homomorphisms.
As mentioned in the previous paragraph, the map $\can$ is an isomorphism.
The map $\alpha$ is surjective, since $s=\coprod s_i$ is a closed subscheme of the affine scheme $S$.
Whence the map $\beta$ is surjective.

For each $s_i \in s$ the curve $\overline X_{s_i}$ is a $k(s_i)$-smooth geometrically irreducible $k(s_i)$-curve of genus $g$.
Whence there exists an integer $n_0$ such that for each $n \geq n_0$ and each $s_i \in s$  has
$H^1(\overline X_{s_i}, \mathcal O((n-1)Y_{s_i}))=0$. Thus
%the sheaf
%$\mathcal O_{\overline X_s}(nY_s)$
%is very ample.
%By the Bertini type theorem
%\cite[Exp. XI, Thm. 2.1]{LNM305}
there exists for any $i$ a section
$t_{1,i}$ of $\mathcal O_{\overline X_{s_i}}(nY_{s_i})$ that does not vanish on $Y_{s_i}$.
By the surjectivity of $\beta$ we may choose a section
$t_1$ of $\mathcal O_{\overline X}(nY)$
such that
$\beta(t_1)|_{\overline X_{s_i}}=t_{1,i}$ for any $i$.
The vanishing locus of $t_1$ does not intersect $Y_s$, whence it does not intersect $Y$.
Clearly, $t_1$ is the desired section of
$\mathcal O_{\overline X}(nY)$. It remains to take for $t_0$ a section of $\mathcal O_{\overline X}(nY)$
with the vanishing locus $nY$.

\end{proof}

\subsection{Proof of the Horrocks type Theorem~\ref{ConstantOnFibreIsConstant}}\label{horrockstypetheorem}
In this section we give a proof of Theorem~\ref{ConstantOnFibreIsConstant}.
This proof is rather standard,
and for the most part follows \cite{R1}. However, our group scheme
$G$ does not come from the ground field $k$. Therefore, we have to
somewhat modify Raghunathan's arguments. We will use the following lemma.
%{\bf We use here notation from Theorem~\ref{ConstantOnFibreIsConstant} and notation from the text
%immediately preceding Theorem~\ref{NewBundle}}.
%Recall that in the Theorem, $k$ is an infinite field, $\mathcal O$ is the semi-local ring of finitely many points on
%a $k$-smooth affine scheme $X$;
%$x= \{x_1,x_2,\dots,x_n\}  \subset \Spec(\mathcal O)$
%is the set of all closed points,
%$l=\prod^r_{i=1}k(x_i)$.
%be the residue field at $x$,
%Also, $G$ is a simple
%simply connected $\mathcal O$-group scheme,
%$G_l = G \otimes_{\mathcal O} l$, \ $\Pro^1:=\Pro^1_{\mathcal O}$, $\Aff^1:= \Aff^1_{\mathcal O}$.

%Let $k$, $\mathcal O$, $x \in Spec(\mathcal O)$, $l$, $G$, $G_l$, $\Pro^1:=\Pro^1_{\mathcal O}$, $\Aff^1_{\mathcal O}$
%be the same as in that Theorem.

\begin{lem}
\label{Horrocks}
Let $W$ be a semi-local irreducible Noetherian scheme over an arbitrary field $k$.
% and let $w_1,w_2, \dots, w_n$ be all its closed points.
Let $H$ and $H'$ be two reductive group schemes over $W$, such that $H$ is a closed $W$-subgroup scheme of $H'$,
and denote by $j: H \hra H^{\prime}$ the corresponding embedding. Denote by $\Pro^1_W$ the projective line over
$W$.

Let $F \in H^1(\Pro^1_W, H)$
be a principal $H$-bundle, and let
$M :=j_*(F) \in H^1(\Pro^1_W, H^{\prime})$
be the corresponding principal $H^{\prime}$-bundle.
If $M$ is a trivial $H^{\prime}$-bundle, then there exists a principal
$H$-bundle $F_0$ over $W$ such that
$pr^*(F_0) \cong F$,
where
$pr: \Pro^1_W \to W$
is the canonical projection.
\end{lem}
\begin{proof}
%It remains to prove Lemma \ref{Horrocks}.
Set $X= H^{\prime}/j(H)$. Locally in the \'{e}tale topology on $W$
this scheme is isomorphic to the $W$-scheme $W \times_{\Spec(k)} H^{\prime}_{0,k}/H_{0,k}$, where
$H_{0,k}$ and $H^{\prime}_{0,k}$ are
the split reductive $k$-group schemes of the same types as $H$ and $H^{\prime}$ respectively.
By results of Haboush
\cite{Hab}
and Nagata
\cite{Na}
(see Nisnevich \cite[Corollary]{Ni1})
the $k$-scheme
$H^{\prime}_{0,k}/H_{0,k}$
is an affine $k$-scheme. Thus $X$ is an affine $W$-scheme.
Consider the long exact sequence of pointed sets
$$1 \to H(\Pro^1_W) \xra{j_*} H^{\prime}(\Pro^1_W) \to X(\Pro^1_W)
\xra{\partial} H^1_{\text{\'{e}t}}(\Pro^1_W, H) \xra{j_*} H^1_{\text{\'{e}t}}(\Pro^1_W, H^{\prime}).$$
Since $j_*(F)$ is trivial, there is $\varphi \in X(\Pro^1_W)$ such that
$\partial (\varphi)= F$.

The $W$-morphism $\varphi: \Pro^1_W \to X$ is a $W$-morphism of a $W$-projective scheme to a $W$-affine scheme.
Thus $\varphi$ is "constant", that is, there exists a section $s: W \to X$ such that
$\varphi= s \circ pr$.
Consider another long exact sequence of pointed sets, this time the one corresponding to the scheme $W$,
and the morphism of the first sequence to the second one induced by the projection $pr$. We get a big commutative diagram.
In particular, we get the following commutative square
\begin{equation}
\label{Representation1}
    \xymatrix{
X(W) \ar[d]_{pr^*_W} \ar[rr]^{\partial} && H^1_{\text{\'{e}t}}(W, H) \ar[d]^{pr^*_W} \\
X(\Pro^1_W) \ar[rr]^{\partial} && H^1_{\text{\'{e}t}}(\Pro^1_W, H).       \\  }
\end{equation}
We have $\pr^*_W(s)=\varphi$. Hence
$$F= \partial (\varphi) = \partial (\pr^*_W(s)) = pr^*_W ( \partial (s)).$$
Setting $F_0=\partial (s)$ we see that
$F= pr^*_W ( F_0)$. The Lemma is proved.

\end{proof}

\begin{proof}[Proof of Theorem~\ref{ConstantOnFibreIsConstant}]
%We use here notation from Theorem~\ref{ConstantOnFibreIsConstant} and notation from the text
%immediately preceding Theorem~\ref{NewBundle}.
It is routine to prove that there is a closed $B^{\prime}$-group scheme embedding
$j: G \hra GL_{N,B^{\prime}}$
for an $N > 0$. By the assumption of the theorem, the $G$-bundle $E$ is trivial on
$\Pro^1_l \subset \Pro^1$. Hence the $GL_{N,B^{\prime}}$-bundle $j_*(E)$ over $\Pro^1$
is trivial over $\Pro^1_l$.

The $B^{\prime}$-group scheme
$GL_{N,B^{\prime}}$
is just the ordinary general linear group.
Thus $j_*(E)$ corresponds to a vector bundle $M$ over $\Pro^1$.
Moreover, this vector bundle is trivial on
$\Pro^1_l$.
Using the equality
$H^1(\Pro^1, \mathcal O_{\Pro^1})=0$
and
\cite[Cor. 4.6.4]{EGAIII1},
we see that $M$ is of the form $M=pr^*(M_0)$
for a vector bundle $M_0$ over $\Spec(B^{\prime})$.
Since $B^{\prime}$ is semi-local, $M_0$ is trivial over $\Spec(B^{\prime})$.
Thus $M$ is trivial on $\Pro^1$.
%By Horrocks' theorem the bundle $M$ is trivial on $\Pro^1 \times_U V$.
Thus
$j_*(F)$ is a trivial $GL_{N,B^{\prime}}$-bundle.

Now, applying Lemma~\ref{Horrocks} to the embedding
$j:G \hra GL_{N,B^{\prime}}$,
we see that $E=pr^*(E_0)$ for some
$E_0 \in H^1(B^{\prime}, G)$.
%Since
%$H=\prod_{\lambda \in \Lambda} GL_{1, A_{\lambda}}$
%and $V$ is semi-local, the bundle $F_0$ is trivial by Hilbert 90 for Azumaya algebras.
%Whence $F$ is trivial as well.
%The isomorphism (\ref{Faddeev})
%shows that
%$(\rho_U)_*(E)$
%is trivial too.
%Now, applying Lemma~\ref{Horrocks} to $W=U$, $j=\rho_U: G \hra R_{V/U}(H)$ and $F=E$,
%we see that there exists a principal $G$-bundle $E_0$ over $U$ such that
%$pr^*_U(E_0)\cong E$.
Theorem~\ref{ConstantOnFibreIsConstant} is proved.

\end{proof}

%%%%%%%%%%%%%%%%%%%%%%%%%%%%%%%%%%%%%%%%%%%%%%%%%%%%%%%%%%%%%


\begin{thebibliography}{EGAIII}

\bibitem[A]{LNM305}
\emph{Artin, M.} Comparaison avec la cohomologie classique:
cas d'un pr\'esch\'ema lisse, in {\it Th\'eorie des topos et cohomologie \'etale des sch\'emas (SGA 4). Tome 3.}
Lect.~Notes Math., vol.~305, Exp. XI, Springer-Verlag, Berlin-New York, 1973.

\bibitem[BT1]{BT65}
\emph{Borel, A.; Tits, J.} Groupes r\'{e}ductifs,
Publ.~Math.~IH\'{E}S 27 (1965), 55--151.

\bibitem[BT2]{BT72}
\emph{Borel, A.; Tits, J.} Compl\'{e}ments \`{a} l'article
``Groupes r\'{e}ductifs'', Publ.~Math.~IH\'ES 41 (1972), 253--276.

\bibitem[BT3]{BT73}
\emph{Borel, A.; Tits, J.} Homomorphismes ``abstraits'' de groupes
algebriques simples, Ann.~Math. 97 (1973), no.~3, 499--571.

\bibitem[Ch]{Ch} \emph{Chernousov, V.} Variations on a theme of groups splitting by a quadratic extension
and Grothendieck-Serre conjecture for group schemes $F_4$ with trivial $g_3$ invariant, Doc. Math., Extra Volume: Andrei A.
Suslin's Sixtieth Birthday (2010), 147--169.


\bibitem[ChP]{ChP} \emph{Chernousov, V.; Panin, I.} Purity of $G\sb 2$-torsors,
C.~R.~Math.~Acad.~Sci.~Paris 345 (2007), no. 6, 307--312.

\bibitem[C-T/O]{C-TO}
\emph{Colliot-Th\'el\`ene, J.-L.; Ojanguren, M.}
Espaces Principaux Homog\`enes Localement Triviaux,
Publ.~Math.~IH\'ES 75 (1992), no.~2, 97--122.


\bibitem[C-T/S]{C-T-S}
\emph{Colliot-Th\'el\`ene, J.-L.; Sansuc, J.-J.}
Principal homogeneous spaces under flasque tori: Applications,
Journal of Algebra 106 (1987), 148--205.


\bibitem[SGA3]{SGA3} \emph{Demazure, M.; Grothendieck, A.} Sch\'emas en groupes,
Lect.~Notes Math., vol. 151--153, Springer-Verlag, Berlin-Heidelberg-New York, 1970.


\bibitem[E]{E} \emph{Eisenbud, D.} Commutative algebra with a view toward algebraic geometry.
Graduate Texts in Mathematics 150, Springer-Verlag, New York, 1995.

\bibitem[FP]{FP} \emph{Fedorov, R.; Panin, I.} A proof of Grothendieck--Serre conjecture on principal bundles over a semilocal regular
ring containing an infinite field, Preprint, April, 2013,
\href{http://www.arxiv.org/abs/1211.2678}{http://www.arxiv.org/abs/1211.2678v2}.

\bibitem[Ga]{Ga} \emph{Gabber, O.} announced and still unpublished.

\bibitem[Gi]{Gille07}  \emph{Gille, Ph.} Le probl\`eme de Kneser-Tits, S\'em. Bourbaki 983
(2007), 983-01--983-39.


\bibitem[Gr1]{Gr1}
\emph{Grothendieck, A.}
Torsion homologique et section rationnalles,
in {\it Anneaux de Chow et applications},
S\'{e}minaire Chevalley, 2-e ann\'{e}e, Secr\'{e}tariat math\'{e}matique,
Paris, 1958.

\bibitem[EGAIII]{EGAIII1}
\emph{Grothendieck, A.}
\'{E}l\'{e}ments de g\'{e}om\'{e}trie alg\'{e}brique (r\'{e}dig\'{e}s avec la collaboration de Jean Dieudonn\'{e}) : III.
\'{E}tude cohomologique des faisceaux coh\'{e}rents, Premi\`ere partie, Publ.~Math.~IH\'ES 11 (1961),
5--167.

\bibitem[EGAIV]{EGAIV2}
\emph{Grothendieck, A.}
\'{E}l\'{e}ments de g\'{e}om\'{e}trie alg\'{e}brique (r\'{e}dig\'{e}s avec la collaboration de Jean Dieudonn\'{e}) : IV.
\'{E}tude locale des sch\'{e}mas et des morphismes de sch\'{e}mas, Seconde partie, Publ.~Math.~IH\'ES 24 (1965),
5--231.


\bibitem[SGA1]{SGA1} %\emph{Hironaka, H.} Exposition II, SGA1.
\emph{Grothendieck, A.} Rev\^etements \'etales et groupe fondamental (SGA 1). Fasc. I: Expos\'es 1 \`a 5.
S\'eminaire de G\'eom\'etrie Alg\'ebrique, 1960/61, Inst. Hautes \'Etudes Sci., Paris, 1963.


\bibitem[Gr2]{Gr2}
\emph{Grothendieck, A.}
Le group de Brauer II, in {\it Dix expos\'{e}s sur la cohomologique de sch\'{e}mas},
Amsterdam, North-Holland, 1968.




%\bibitem[H]{H}
%\emph{Harder, G.} Halbeinfache Gruppenschemata \"uber Dedekindringen,
%Invent.~Math. 4, 165--191 (1967).

\bibitem[Hab]{Hab} \emph{Haboush, W.J.} Reductive groups are geometrically reductive, Ann. Math. 102 (1975), no.~1, 67--83.


\bibitem[Ha]{Ha} \emph{Hartshorne, R.} Algebraic geometry. Graduate Texts in Mathematics 52, Springer-Verlag, New York-Heidelberg, 1977.


\bibitem[Mo]{Mo} \emph{Moser, L.-F.} Rational triviale Torseure und die Serre-Grothendiecksche Vermutung, Diplomarbeit,
2008, \href{http://www.mathematik.uni-muenchen.de/~lfmoser/da.pdf}{http://www.mathematik.uni-muenchen.de/~lfmoser/da.pdf}.


\bibitem[Mu]{Mu} \emph{Mumford, D.} Abelian Varieties, Oxford University Press, Oxford, 1970.

\bibitem[Na]{Na} \emph{Nagata, M.} Invariants of a group in an affine ring, J. Math. Kyoto Univ. 3 (1964), no.~3, 369--377.

\bibitem[Ni1]{Ni1} \emph{Nisnevich, E.A.} Affine homogeneous spaces and finite subgroups of arithmetic groups
over function fields,  Functional Analysis and Its Applications 11 (1977), no.~1, 64--66.



\bibitem[Ni2]{Ni} \emph{Nisnevich, Y.}
Rationally Trivial Principal Homogeneous Spaces and Arithmetic of
Reductive Group Schemes Over Dedekind Rings,
C.~R.~Acad.~Sci.~Paris, S\'erie I, 299 (1984), no.~1, 5--8.



\bibitem[OP1]{OP2}
\emph{Ojanguren, M.; Panin, I.}
A purity theorem for the Witt group,
Ann. Sci. Ecole Norm. Sup. (4) 32 (1999), no.~1, 71--86.


\bibitem[OP2]{OP1} \emph{Ojanguren, M.; Panin, I.}
Rationally trivial hermitian spaces are locally trivial,
Math.~Z. 237 (2001), 181--198.


\bibitem[OPZ]{OPZ} \emph{Ojanguren, M.; Panin, I.; Zainoulline, K.}
On the norm principle for quadratic forms,
J.~Ramanujan Math.~Soc. 19 (2004), no.~4, 1--12.

\bibitem[PS]{PS} \emph{Panin, I.; Suslin, A.}
On a conjecture of Grothendieck concerning Azumaya algebras,
St.~Petersburg Math.~J.  9 (1998), no.~4, 851--858.


\bibitem[P1]{Pa1} \emph{Panin, I.} A purity theorem for linear algebraic groups, Preprint,
2005, \href{http://www.math.uiuc.edu/K-theory/0729}{http://www.math.uiuc.edu/K-theory/0729}.

\bibitem[P2]{Pa2}
\emph{Panin, I.}
On Grothendieck---Serre's conjecture concerning
principal $G$-bundles over reductive group schemes:II,
Preprint, April, 2013, \href{http://www.arxiv.org/0905.1423v3}{http://www.math.org/0905.1423v3}.

\bibitem[PPeSt]{PPS}
\emph{Panin, I.; Petrov, V.; Stavrova, A.}
On Grothendieck--Serre's for simple adjoint group schemes of types $E_6$ and $E_7$,
Preprint, 2009, \href{http://www.math.uiuc.edu/K-theory/}{http://www.math.uiuc.edu/K-theory/}.

\bibitem[PeSt1]{PSt}
\emph{Petrov V.; Stavrova A.} Elementary subgroups in isotropic
reductive groups, St. Petersburg Math. J. 20 (2009), no.~4, 625--644.


\bibitem[PeSt2]{PS-f4} \emph{Petrov V.; Stavrova, A.}
Grothendieck---Serre conjecture for groups of type $F_4$ with trivial $f_3$
invariant, Preprint, 2009,  \href{http://www.mathematik.uni-bielefeld.de/LAG/man/374.html}{http://www.mathematik.uni-bielefeld.de/LAG/man/374.html}.


\bibitem[Po]{P}
\emph{Popescu, D.}
General N\'eron desingularization and approximation,
Nagoya Math.~J. 104 (1986), 85--115.




\bibitem[R1]{R1}
\emph{Raghunathan, M.S.} Principal bundles admitting a rational section,
Invent.~Math. 116 (1994), no. 1--3, 409--423.

\bibitem[R2]{R2}
\emph{Raghunathan, M.S.} Erratum: Principal bundles admitting a rational
section, Invent.~Math. 121 (1995), no.~1, 223.

\bibitem[RR]{RR}
\emph{Raghunathan, M.S.; Ramanathan, A.} Principal bundles on the
affine line, Proc.~Indian Acad.~Sci., Math.~Sci. 93 (1984), 137--145.

%\bibitem[R]{R}
%\emph{Rost, M.} Durch Normengruppen definierte birationale Invarianten.
%C.~R.~Acad.~Sci.~Paris 10 (1990), ser.I, 189--192.

%\bibitem[Q]{Q}
%\emph{Quillen, D.}
%Higher algebraic $K$-theory I, in Algebraic K-Theory I.
%Lect.~Notes Math. 341, Springer-Verlag, 1973.

%\bibitem[Se]{Se}
%\emph{Serre, J-P.}
%Cohomologie Galoisienne. Springer-Verlag (1964).

%\bibitem[SV]{SV}
%\emph{Suslin, A.; Voevodsky, V.}
%Durch Normengruppen definierte birationale Invarianten???
%%%%%%%% I don't believe it
%C.~R.~Acad.~Sci.~Paris 10, ser.I, 189--192 (1990).

\bibitem[Se]{Se}
\emph{Serre, J.-P.}
Espaces fibr\'{e}s alg\'{e}briques, in
{\it Anneaux de Chow et applications},
S\'{e}minaire Chevalley, 2-e ann\'{e}e, Secr\'{e}tariat math\'{e}matique,
Paris, 1958.


\bibitem[Sw]{Sw}
\emph{Swan, R.G.}
N\'eron---Popescu desingularization, Algebra and Geometry (Taipei, 1995),
Lect.~Algebra Geom. 2,
Internat.~Press, Cambridge, MA, 1998, 135--192.

\bibitem[T]{T}
\emph{Tits, J.} Algebraic and abstract simple groups,
Ann.~Math. 80 (1964), no. 2, 313--329.


\bibitem[Z]{Z}
\emph{Zainoulline, K.V.}
On Grothendieck's conjecture on principal homogeneous spaces for some classical algebraic groups,
St. Petersburg Math. J. 12 (2001), no.~1, 117--143.

\bibitem[Vo]{Vo}
\emph{Voevodsky, V.}
Cohomological theory of presheaves with transfers, in {\it
Cycles, Transfers, and Motivic Homology Theories},
Ann.~Math.~Studies, 2000, Princeton University Press.

\end{thebibliography}
\end{document}